\documentclass[10pt,oneside,reqno]{amsart}

  \usepackage[utf8]{inputenc}
  \usepackage[T1]{fontenc}
  \usepackage[russian,ngerman,english]{babel}
  \usepackage{lmodern}

  \usepackage{hyperref}
  \usepackage{amsmath}
  \usepackage{amsthm}
  \usepackage{amsfonts}
  \usepackage{amssymb}
  \usepackage{amscd}
  \usepackage{amsbsy}

  \usepackage{pdfpages}
  \usepackage{fancyhdr}
  \usepackage{geometry}
  \usepackage{setspace}
  \usepackage{xcolor}
  \usepackage{pifont}

  \usepackage{graphicx}
  \usepackage{tabularx}
  \usepackage{epsfig}
  \usepackage[all]{xy}
  \usepackage{tikz}
  \usetikzlibrary{matrix}

  \usepackage{fixmath}

  \usepackage{appendix}

  \usepackage{enumerate}

  \usepackage[sort, numbers]{natbib}
  \usepackage{bibentry}

  \newtheorem{theorem}{Theorem}[section]

  \newtheorem{lemma}[theorem]{Lemma}
  \newtheorem{proposition}[theorem]{Proposition}
  
  \newtheorem{corollary}[theorem]{Corollary}
  
  \theoremstyle{definition}
  \newtheorem{definition}[theorem]{Definition}

  \newtheorem*{remark}{Remark}
  
  \newtheorem{problem}[theorem]{Problem}

  \numberwithin{equation}{section}

  \newcommand{\Z}{{\mathbb Z}}

  \newcommand{\R}{{\mathbb R}}
  \renewcommand{\C}{{\mathbb C}}
  \newcommand{\D}{{\mathbb D}}
  \newcommand{\T}{{\mathbb T}}

  \newcommand{\E}{{\mathsf E}}

  \newcommand{\PSL}{\mathrm{PSL}}

  \newcommand{\supp}{{\operatorname{supp}}}

  \newcommand{\sgn}{\operatorname{sgn}}

  \newcommand{\cD}{{\mathcal{D}}}
  
  \newcommand{\cF}{{\mathcal{F}}}

  \newcommand{\cL}{{\mathcal{L}}}

  \newcommand{\cP}{{\mathcal{P}}}

  \newcommand{\cV}{{\mathcal{V}}}

  \newcommand{\bb}{\mathbf{b}}
  \newcommand{\ba}{\mathbf{a}}

    \newcommand{\bz}{\mathbf{z}}
  
  \newcommand{\bc}{\mathbf{c}}
  \newcommand{\bC}{\mathbf{C}}

  \renewcommand{\G}{{\Gamma}}
  \newcommand{\e}{{\varepsilon}}

  \newcommand{\z}{\zeta}
  \newcommand{\g}{\gamma}
  
  \renewcommand{\l}{\lambda}
  \renewcommand{\a}{\alpha}

  \renewcommand{\d}{\delta}

  \newcommand{\spa}{\operatorname{span}}
  \renewcommand{\Re}{\operatorname{Re}}
  
  \renewcommand{\Cap}{\operatorname{Cap}}
  \newcommand{\dist}{\operatorname{dist}}
  
  \newcommand{\dx}{\mathrm{d}x}

  \newcommand{\wlim}{\operatorname*{w-lim}}

  \newcommand{\bbD}{\mathbb{D}}
  
  \newcommand{\bbZ}{\mathbb{Z}}
  
  \newcommand{\bbR}{\mathbb{R}}
  \newcommand{\bbC}{\mathbb{C}}
  
  \newcommand{\bbT}{\mathbb{T}}

  \allowdisplaybreaks

  \title[Asymptotics of Chebyshev rational functions]{Asymptotics of Chebyshev rational functions with respect to subsets of the real line}

\thanks{B.E.\ was supported by Austrian Science Fund FWF, project no: J 4138-N32.}
\thanks{M.L.\ was supported in part by NSF grant DMS--1700179.}
\thanks{G.Y.\ was supported in part by NSF grant DMS--1745670.}

\author{Benjamin Eichinger, Milivoje Luki\'c, Giorgio Young}

\address{Institute of Analysis, Johannes Kepler University Linz, 4040 Linz, Austria.}
\email{benjamin.eichinger@jku.at}
\address{Department of Mathematics, Rice University MS-136, Box 1892,
Houston, TX 77251-1892, USA.}
\email{milivoje.lukic@rice.edu}
\address{Department of Mathematics, Rice University MS-136, Box 1892,
Houston, TX 77251-1892, USA.}
\email{gfy1@rice.edu}

\begin{document}
\begin{abstract}
There is a vast theory of Chebyshev and residual polynomials and their asymptotic behavior. The former ones maximize the leading coefficient and the latter ones maximize the point evaluation with respect to an $L^\infty$ norm.
We study Chebyshev and residual extremal problems for rational functions with real poles with respect to subsets of $\overline{\R}$. We prove root asymptotics under fairly general assumptions on the sequence of poles. Moreover, we prove Szeg\H o--Widom asymptotics for sets which are regular for the Dirichlet problem and obey the Parreau--Widom and DCT conditions.
\end{abstract}

\maketitle

\section{Introduction}

Chebyshev polynomials are extremal polynomials with respect to the supremum norm on a compact set $\E$. First discovered with explicit formulas for the set $\E = [-1,1]$, see \cite{Chebyshev,Chebyshev2}, a general theory has developed for more general sets $\E$, with important classical and modern developments \cite{ChrSimZin,ChriSiZiYuDuke,EichYuSbornik,SodYud93,Widom69}. Aspects of this theory have been extended to the setting of residual polynomials \cite{ChrSimZin} (which are extremizers with respect to a point evaluation rather than leading coefficient) and to the setting of Chebyshev rational functions with poles in $\overline{\bbR} = \bbR \cup \{\infty\}$ \cite{LukashovJAT98}.

To state the problems precisely, we make the following definitions. For $\bc \in \overline{\bbR}$ we denote
\begin{align*}
r(z,\bc)=\begin{cases}
\frac{1}{\bc-z},\quad &\bc\neq\infty,\\
z,&\bc=\infty.
\end{cases}
\end{align*}
We fix a compact proper subset $\E \subset \overline{\bbR}$ containing infinitely many points. Connected components of $\overline{\bbR} \setminus \E$ are called gaps of $\E$. We fix a sequence of poles $\bC = (\bc_k)_{k=1}^\infty$ with $\bc_k \in \overline{\bbR} \setminus \E$. The sequence $\bC$ can have repetitions, which are used to designate multiplicity: we consider the spaces of rational functions $\cL_n$ defined as
\begin{align}\label{eq:rationalspace}
\cL_n=\left\{\frac{P(z)}{R_n(z)}:\ P\in\cP_n\right\},
\end{align}
where $\cP_n$ denotes the set of polynomials of degree at most $n$ and
\begin{align}\label{eq:denominator}
R_n(z)=\prod_{\substack{1 \le k \le n \\ \bc_k \neq \infty}} (z-\bc_k).
\end{align}
Of course, the spaces $\cL_n$ could also be defined iteratively, by
\begin{align*}
\cL_n=\spa\left\{r(z,\bc_n)^{d_n}\right\} \oplus \cL_{n-1},\quad \cL_0=\{1\},
\end{align*}
where $d_n$ denotes the multiplicity of the pole $\bc_n$ up to that point,
\begin{align*}
d_n= \sum_{\substack{1\le k\le n \\ \bc_k = \bc_n}} 1.
\end{align*}
Let $\| \cdot \|_\E$ denote the supremum norm on $\E$. We consider the two related extremal problems:
  \begin{problem}[Chebyshev extremal problem] \label{prob1}
  \begin{align}\label{eq:Linftyextremal}
m_n(\bc_n) := \sup\{\Re\l_n: \exists F_n \in \cL_n \text{ such that } \|F_n\|_\E\leq 1 \text{ and } F_n - \l_n r(\cdot,\bc_n)^{d_n}\in \cL_{n-1}  \}.
  \end{align}
\end{problem}

  \begin{problem}[Residual extremal problem] \label{prob2}
  For $x_* \in\overline{\R}\setminus (\E\cup\{\bc_k: 1\leq k\leq n\})$,
  	\begin{align}\label{eq:residual}
  	m_n(x_*):=\sup\{\Re F_n(x_*): F_n\in\cL_n,\ \|F_n\|_\E\leq 1 \}.
  	\end{align}
    \end{problem}

If $\bc_k=\infty$ for all $k$, Problem~\ref{prob1} is the standard extremal problem for Chebyshev polynomials on $\E$. For this reason we refer to $\l_n$ still as the leading coefficient.
 Whereas the Chebyshev extremal problem maximizes the leading coefficient at the pole $x_* = \bc_n$, the residual extremal problem maximizes the value at a point $x_*$ which is not a pole. We will use the notation $x_*$ for both problems when convenient.

For both problems, an extremal function exists (i.e., the supremum is a maximum) and is unique (see Section~\ref{sectionPropertiesFixedn}). The goal of this paper is to study the extremal functions $F_n$ and their asymptotics as $n\to\infty$.

Problems~\ref{prob1}/\ref{prob2} have a conformal invariance with respect to the group $\PSL(2,\bbR)$ of $\overline{\bbR}$-preserving, orientation-preserving M\"obius transformations. This conformal invariance is obfuscated by the use of polynomials in the definitions \eqref{eq:rationalspace} and \eqref{eq:denominator}, but can be made explicit in the language of divisors. Divisors on the Riemann sphere $\overline{\bbC} = \bbC \cup \{\infty\}$ are elements of the free Abelian group over $\overline{\bbC}$. They can be implemented as formal sums or as functions $D: \overline{\bbC} \to \bbZ$ which take nonzero values only at finitely many points; we will find the second interpretation notationally convenient. The degree of $D$ is the integer $\deg D = \sum_{z} D(z)$, and the divisor $D$ is integral if $D(z) \ge 0$ for all $z$.  We also write $D_1\leq D_2$, if $D_2-D_1$ is integral and denote by $\supp D=\{z\in\overline\C: D(z)\neq 0\}$ the support of $D$. In particular, for a meromorphic nonconstant function $f: \overline{\bbC}\to \overline{\bbC}$, we denote its polar divisor by $(f)_\infty$; the polar divisor assigns to each pole the multiplicity of that pole, and takes zero values elsewhere. Similarly, for $w\in \bbC$, we define $(f)_w = (1/(f-w))_{\infty}$. The value $\deg (f)_w$ is independent of $w$ and corresponds to the degree of $f$.  We also follow the convention to set $(f)_w= 0$, if $f$ is a constant. For any $n$, we define the divisor $D_n^\infty$ by
\begin{align}\label{def:Dninfty}
	D_n^\infty(\bc)=\#\{k:\ \bc_k=\bc,\ 1\leq k\leq n\}.
\end{align}
In other words, in the functional interpretation, $D_n^\infty = \sum_{k=1}^n \chi_{\{\bc_k\}}$.  Note that by definition $\deg D_n^\infty=n$. Any integral divisor $D$ with degree $n$ generates a $n+1$ dimensional vector space
 \begin{align}\label{eq:LDivisor}
 \cL(D)=\{ f: \overline{\bbC} \to \overline{ \bbC} \mid f\text{ is meromorphic and } (f)_\infty \le D \},
 \end{align}
and the definition \eqref{eq:rationalspace} is equivalent to
\begin{align}\label{eq:LnDivisor}
\cL_n = \cL(D_n^\infty).
\end{align}

Now Problems~\ref{prob1}, \ref{prob2} can be unified as follows:

 \begin{problem} \label{prob3}
For a real integral divisor $D_n^\infty$ with $\deg D_n^\infty = n$ containing only points in $\overline{\bbR} \setminus \E$, and a point  $x_* \in\overline{\R}\setminus \E$, denote $d_n = D_n^\infty(x_*)$ and $\cL_n=\cL(D_n^\infty)$ and find
  	\begin{align}\label{eq:ExtremalProblem}
  	m_n(x_*):=\sup\{\Re \lim_{x\to x_*} \frac{F_n(x)}{r(x,x_*)^{d_n}} :  F_n\in\cL_n,\ \|F_n\|_\E\leq 1 \}.
  	\end{align}
\end{problem}
The Chebyshev problem corresponds to $d_n > 0$ (up to a permutation of $\bc_1,\dots, \bc_n$) and the residual problem corresponds to $d_n = 0$. Throughout this paper, we work in the general setting of Problem~\ref{prob3}.

In order to state our results in a conformally invariant form, we use the following language:

\begin{definition}
For a sequence $(t_j)_{j=0}^m$ in $\overline{\bbR}$ with $m\ge 2$, we say that the sequence is cyclically ordered if it has no repetitions and there exists $f \in \PSL(2,\bbR)$ such that $f(t_0) = \infty$ and $f(t_1) < f(t_2) < \dots < f(t_{m})$. We will also use cyclic interval notation: for distinct $a, b \in \overline{\bbR}$, we denote
\[
(a,b) = \{ c \mid (a,c,b)\text{ is cyclically ordered}\}, \qquad [a,b] = \{a,b\} \cup (a,b).
\]
\end{definition}

This gives a well-defined cyclic order, since $\PSL(2,\bbR)$ transformations preserve orientation on $\overline{\bbR}$.

Chebyshev polynomials for subsets of $\bbR$ have many universal properties; the Chebyshev alternation theorem compresses all these properties in a way that uniquely characterizes the extremizer. Namely, a polynomial $P_n$  of degree n so that $\|P_n\|_\E\leq 1$ has a maximal set of alternation points if there are $n+1$ points $x_1<\dots<x_{n+1}$, $x_i\in\E$, so that
\begin{align}\label{eq:alteCheby}
P_n(x_j)=(-1)^{n+1-j}.
\end{align}
Then $P_n$ is the Chebyshev polynomial for the set $\E$, if and only if it has a maximal set of alternation points.
One way of viewing the alternation theorem is the following. The Chebyshev polynomial, $T_n$, for $\E$ has $n$ real and simple zeros and between each of them there should be an alternation point, which gives $n-1$ of them and then there should be one at each gap edge of the extremal gap (in this case the one containing $\infty$) which sums up to $n+1$ points of alternation. In particular $x_1$ and $x_{n+1}$ will always be counted, because of the natural order of $\R$.  Similarly, residual polynomials have an alternation theorem, which relies on a notion of an $x_*$ alternation set \cite{ChrSimZin}. Furthermore, by \cite{ChrSimZin}, in the polynomial case, such a set characterizes the residual polynomial: $P_n$ is the residual polynomial for the set $\E$ if and only if $\| P_n\|_{\E}\leq 1$ and $P_n$ has an $x_*$ alternation set.

In the setting of rational functions the counting is essentially more delicate, and the relative ordering of the poles and alternation points play an important role. The reason for this is that if between two zeros there is a gap with a pole $\bc_j$, then the sign at the next gap edge depends on the parity of the pole. This makes it necessary to define the following \textit{sign function}:
\begin{align*}
S_n(x)=    \sum_{\substack{1 \le k \le n \\ \bc_k \neq x_*}} \chi_{[x_*,\bc_k)}(x) = \sum_{\bc\in \overline \R\setminus\{x_*\}}D^\infty_n(\bc)\chi_{[x_*,\bc)}(x).
\end{align*}

Recall that a function $F$ is called real if for all $z\in\C$, $\overline{F(\overline{z})}=F(z)$.

\begin{definition}
For a real function $F \in\cL_n$ with $\|F\|_\E\leq 1$, a set of distinct points $x_1, \dots, x_m \in \E$ such that the sequence $(x_*, x_1, \dots, x_{m})$ is cyclically ordered and satisfies the following alternation property
\begin{align}
F(x_j)=(-1)^{m-j-S_n(x_j)}
\end{align}
for all $j=1,\dots,m$ is called an alternation set. We say that $F$ has a maximal alternation set if $m=n+1$.
\end{definition}

It should be noted that the notion of alternation set depends on the function $F$, the class $\cL_n$, and the reference point $x_*$.

\begin{theorem}[Alternation theorem] \label{thmAlternation}
A real function $F \in \cL_n$ with $\lVert F \rVert_\E \le 1$ is an extremal function if and only if it has a maximal alternation set.
\end{theorem}

These results generalize standard results from the polynomial case: in the Chebyshev polynomial case, $S_n(x)\equiv 0$, and in the residual polynomial case,  $S_n$ has one jump which may or may not affect the alternation criterion, depending on degree. The case of Chebyshev rational functions was also previously formulated in \cite{LukashovJAT98}. In all the cases previously considered in the literature, the extremizer is seen to be nonconstant. However, in the setting of residual rational functions, the extremizer can be a constant function, and the alternation theorem lets us characterize when this happens:

\begin{theorem} \label{thmConstant}
The extremal function $F_n$ is constant if and only if the divisor $D_n^\infty$ is of the form \eqref{def:Dninfty} for points $\bc_1,\dots, \bc_n$ such that the points $x_*, \bc_1, \bc_2, \dots, \bc_n$ are in $n+1$ distinct gaps of $\E$.
\end{theorem}

In particular, for the Chebyshev problem, $x_* = \bc_n$ so $F_n$ is always nonconstant.

These results will be proved in Section~\ref{sectionPropertiesFixedn}, along with additional properties of $F_n$ and its zeros. Let us assume that $F_n$ is not constant and recall that $(F_n)_\infty \le D_n^\infty$;  we call a point $x$  a ``generalized zero'' of $F_n$ if either $(F_n)_0(x)>0$ or if
\begin{align*}
D_n^\infty(x)-(F_n)_\infty(x) >0.
\end{align*}
Thus, this notion includes both actual zeros of $F_n$ and places where there is a reduction in the order of the pole compared to the maximal allowed order.  These generalized zeros are precisely counted by the divisor
\begin{align*}
D_n^0 :=(F_n)+D_n^\infty=(F_n)_0+ D_n^\infty-(F_n)_\infty.
\end{align*}

Since an alternation set is on $\E$, note that changing $x_*$ through a single gap only changes the alternation conditions up to an overall $j$-independent $\pm 1$ factor. Therefore, up to $\pm$ sign, the extremizer $F_n$ for Problem~\ref{prob3} is unchanged as $x_*$ varies through a single gap of $\E$. Thus, $F_n$ should be regarded as an extremal function of a gap, rather than of a single point. In particular, the Chebyshev extremizer for Problem~\ref{prob1} is the same (up to $\pm$ sign) as the residual extremizer for Problem~ \ref{prob2} for any $x_*$ in the gap containing $\bc_n$. Moreover, $F_n$ might even be extremal for more than one gap. This phenomenon is already known for the so-called Widom maximizer defined below, and is the content of the following corollary.

\begin{corollary}\label{cor:gapchange}
Let $F_n$ be an extremal function  for $x_*\in (\ba,\bb)$.
If  $(\ba_j,\bb_j)$ is a gap such that $|F_n(\ba_j)|=|F_n(\bb_j)|=1$ and $D_n=0$ on $(\ba_j,\bb_j)$, then up to a $\pm 1$ factor, $F_n$ is an extremal function for any $ x_*^j\in (\ba_j,\bb_j)$.
\end{corollary}

From $\deg (F_n)_0= \deg (F_n)_\infty$ it follows that
\begin{align}\label{eq:degDivZero}
\deg D_n^0=\deg D_n^\infty=n
\end{align}
so we can define the normalized pole counting measure
\begin{equation}\label{7oct2}
\mu_n:= \frac 1n \sum_{\bc} D_n^\infty(\bc) \delta_{\bc}
\end{equation}
and normalized generalized zero counting measure
\begin{equation}\label{nun}
\nu_n:= \frac 1n \sum_{\bc} D_n^0(\bc) \delta_{\bc} .
\end{equation}

In Section 3, we consider the asymptotics of the extremal rational functions as $n \to \infty$, extending results about root asymptotics from the polynomial setting.  For a sequence of divisors $D_n^\infty$ as in Problem \ref{prob3} we define
\begin{align*}
K_\bC=\overline{\bigcup_{n\geq 1}\supp D_n^\infty}
\end{align*}
and assume that
\begin{align}\label{7oct1}
K_{\bC}\cap\E=\emptyset
\end{align}
and that that the sequence has a limiting distribution, i.e., that the normalized pole counting measures $\mu_n$ have a weak limit $\mu$ in the topology dual to $C(\overline{\bbR})$. A similar combination of assumptions, but with poles away from the convex hull of $\E$, is used in \cite[Chapter 6]{StahlTotik92} to study rational interpolation. Some of our current work mirrors our work for orthogonal rational functions \cite{EichLukYou}, but that work required a periodic sequence of poles. In this sense, in addition to studying a different extremal problem, our current setting is more general. To the best of our knowledge all previous works also assumed that the sequence of divisors $D_n^\infty$ is monotonic. Let further $(x_n^*)_{n=0}^\infty$ be a sequence in $\overline \bbR\setminus \E$ which does not accumulate on $\E$.

The behavior of $\log \lvert F_n\rvert$ is governed by the zero and pole distributions.  This corresponds to two Riesz representations, with $\log \lvert F_n \rvert$ superharmonic (respectively, subharmonic) away from the set of zeros (respectively, poles). The limiting pole distribution $\mu$ directly determines the root asymptotics of the functions $F_n$ and the limiting zero distribution.

We assume that $\E$ is not a polar set, i.e., the domain $\Omega = \overline{\bbC} \setminus \E$ is Greenian, and we denote by $G(z,w) = G_\E(z,w)$ the Green function and by $\omega_\E(dz,x)$ harmonic measure for this domain.

\begin{theorem}[Root asymptotics] \label{thm:RootAsymptotics}
Assume that $\E$ is not a polar set, \eqref{7oct1} holds,  the measures $\mu_n$ converge weakly to $\mu$ in the topology dual to $C(\overline{\bbR})$ and $(x_n^*)_{n=0}^\infty$ be a sequence in $\overline\R\setminus \E$ not accumulating on $\E$. Then uniformly on compact subsets of $\bbC\setminus\bbR$,
\begin{align*}
\lim\limits_{n\to\infty}\frac{1}{n}\log \lvert F_n(z) \rvert = \int G_\E(z,x)d\mu(x).
\end{align*}
Moreover,
\[
\wlim_{n\to\infty} \nu_n = \int \omega_\E(dz,x) d\mu(x).
\]
\end{theorem}

Our proof of root asymptotics relies on an explicit representation of $F_n$ in terms of the so-called $n$-extension $\E_n=F_n^{-1}([-1,1])$. Representations of this type appear for instance in \cite{ChrSimZin,SodYud93}. In particular, using $\E\subset\E_n$ and monotonicity of the Green function, we obtain a Bernstein-Walsh type upper bound for $F_n$ in terms of the Green functions $G_\E(z,\bc)$. This is the major difference between the $L^2$ and the $L^\infty$ setting. In the $L^2$ setting \cite{EichLukYou} an asymptotic upper bound is equivalent to Stahl--Totik regularity of the measure, whereas in the $L^\infty$ setting this bound holds for any $n$.

 As in \cite[Corollary 1.2]{Simon07}, this can be used to describe the behavior of the leading coefficient.

Theorem~\ref{thm:RootAsymptotics} generalizes known polynomial results, which correspond to the degenerate pole distribution $\mu = \delta_\infty$. Another notable case, related to \cite{EichLukYou}, is of a $p$-periodically repeating sequence of poles $\mu= \frac 1p \sum_{j=1}^p \delta_{\bc_j}$.

In Section \ref{sec:SzegoWidom}, we prove so-called Szeg\H o-Widom asymptotics for $F_n$. To the best of our knowledge, all previous results are only for polynomial extremal problems. Let $\Omega$ be a domain in $\overline{\bbC}$ which contains $\infty$ and  $\E=\partial \Omega$ be an analytic Jordan curve, $T_n$ the associated Chebyshev polynomial  and $B_\E$ denote the Riemann map that maps $\Omega\to\D$ and $B_\E(\infty)=0$, normalized so that $\lim\limits_{z\to\infty}zB_\E(z)>0$.  Faber \cite{Fab20} showed that uniformly on compact subsets of  $\Omega$
\begin{align}\label{eq:Faber}
\lim_{n\to\infty}T_nB_\E^n=1.
\end{align}
In his landmark paper \cite{Widom69}, Widom generalized this notion to multiply connected domains. In the following let $\Omega$  be a domain in $\overline{\bbC}$ which contains $\infty$ so that $\E=\partial\Omega$ is not polar. We will describe the type of results for multiply connected domains, but refer the reader for the precise definitions and statements to Section \ref{sec:SzegoWidom}. The correct analog for the Riemann map for multiply connected domains is the so-called complex Green function
\begin{align}\label{eq:ComplexGreen}
B_\E(z,\infty)=e^{-G_\E(z,\infty)-i\widetilde{G_\E(z,\infty)}},
\end{align}
where $\widetilde{G_\E(z,\infty)}$ denotes the harmonic conjugate of $G_\E(z,\infty)$. To be more precise, since $G_\E(z,\infty)$ is harmonic, $B_\E(z,\infty)$ is first defined locally and then using the monodromy theorem \cite[Theorem 11.2.1]{SimonBasicCompAna} extended to a global multivalued analytic function in $\Omega$. Due to the multivaluedness of $B_\E$, one cannot expect that $B_\E^nT_n$ converges to a single analytic function as in \eqref{eq:Faber}. For this reason, Widom considered a related character automorphic extremal problem. Let $z_0\in\Omega$ and let $\pi_1(\Omega,z_0)$ denote the fundamental group of $\Omega$ with basepoint fixed at $z_0$, and $\pi_1(\Omega)^*$  the group of unitary characters of $\pi_1(\Omega,z_0)$; that is, group homomorphisms from $\pi_1(\Omega,z_0)$ into $\T := \R/\Z$.  If $F$ is an analytic function on $\Omega$, then we call $F$ ($\pi_1(\Omega)^*$-) character-automorphic with character $\a$, if
\begin{align*}
F\circ\tilde\g =e^{2\pi i \a(\tilde\g)}F,\quad \forall\tilde\g\in\pi_1(\Omega,z_*).
\end{align*}
Let $H^\infty_\Omega(\a)$ denote the space of analytic character-automorphic functions, $F$, in $\Omega$ which are uniformly bounded, i.e.,
\begin{align}\label{eq:HinftyNorm}
\|F\|_{\Omega}:=\sup_{z\in\Omega}|F(z)|<\infty.
\end{align}
In his `69 paper \cite{Widom69}, Widom considered the extremal problem
\begin{align}\label{eq:HinftyProblem}
\sup\{\Re F(x_*): F\in H^\infty_\Omega(\a), \|F\|_{\Omega}\leq 1\}
\end{align}
under the assumption that $\E$ is a finite union of $C^2$ Jordan curves and arcs and showed existence and uniqueness of the extremizer; let us call this the Widom maximizer. Let $\chi_n$ denote the character of $B_\E^n$ and $W_n$ the Widom maximizer with character $\chi_n$ for the extremal point $x_*=\infty$. If $\E$ is the finite union of $C^2$ Jordan curves, Widom showed that uniformly on compact subsets of $\Omega$
\begin{align*}
B_\E^nT_n-W_n\to 0.
\end{align*}
If such type of convergence holds, we say $T_n$ has Szeg\H o-Widom asymptotics. The cases of arcs turned out to be essentially harder and for non-real problems only very simple cases such as one arc of the unit circle \cite{EichingerJAT17} are known. If $\E\subset\R$ the situation is essentially better, since in this case there are many symmetry properties, which manifests in the fact that the extremal function is real and allows for the explicit representation of the type we will derive in \eqref{eq:extremalFunctionRep}. If $\E$ is a finite union of intervals Christiansen, Simon and Zinchenko \cite{ChrSimZin} showed that $T_n$ has Szeg\H o-Widom asymptotics.  In 1971 Widom \cite{Widom71} also showed that \eqref{eq:HinftyProblem} has a non-trivial solution as long as $\Omega$ is of Parreau--Widom type. We will define this notion in Section \ref{sec:SzegoWidom}, but mention at this place that it also includes infinitely connected domains. Recently Christiansen, Simon, Yuditskii and Zinchenko \cite{ChriSiZiYuDuke} proved Szeg\H o-Widom asymptotics for $T_n$ if $\E\subset\R$ such that $\Omega$ is a regular Parreau--Widom domain with Direct Cauchy theorem and this was later also proved under the same assumptions for residual polynomials \cite{ChrSimZin}.

We point out that
\begin{align*}
(T_n)_\infty=n(B_\E(\cdot,\infty))_0,
\end{align*}
which makes $B_\E^nT_n$ analytic and in fact a normal family. Since by definition
\begin{align*}
(F_n)_\infty\leq D_n^\infty,
\end{align*}
 in our setting $B_\E^n$ should be substituted by the product of complex Green functions associated to the divisor $D_n^\infty$, i.e.
\begin{align}\label{eq:BlaschkeProduct}
B_\E^{(n)}(z)=e^{i\phi_n}\prod_{\bc}D_n^\infty(\bc)B_\E(z,\bc),
\end{align}
where
\begin{align}\label{eq:ComplexGreenGeneral}
B_\E(z,\bc)=e^{-G_\E(z,\bc)-i\widetilde{G_\E(z,\bc)}}
\end{align}
and the phase will be specified in Section \ref{sec:SzegoWidom}. With this modification we prove:
\begin{theorem}
Let $\Omega=\overline{\bbC}\setminus\E$ be a regular Parreau--Widom domain so that the Direct Cauchy theorem holds in $\Omega$. Assume further that  \eqref{7oct1} holds,  the measures $\mu_n$ converge weakly to $\mu$ in the topology dual to $C(\overline{\bbR})$ and $(x_n^*)_{n=0}^\infty$ be a sequence in $\overline\R\setminus \E$ without accumulation points in $\E$. Then $F_n$ admits Szeg\H o-Widom asymptotics.
\end{theorem}

We want to highlight that this generalizes the known results in several ways. First of all, polynomials correspond to the case that $D_n^\infty=n\chi_{\{\infty\}}$ and so the class of functions that we allow is more general. Secondly, we allow for a sequence of extremal points $x_n^*$, which in particular means that depending on $n$, $F_n$ might be a residual or a Chebyshev maximizer.

\section{Properties of the extremal rational functions}  \label{sectionPropertiesFixedn}

In this section we study the extremal functions for fixed $n$. Let us begin by acknowledging that their existence follows by usual arguments. Namely, the leading coefficient $\lambda_n$ and the value $F_n(x_*)$ are continuous functions of polynomial coefficients of $F_n R_n$. Since $\cL_n$ is finite-dimensional,  the norm $\lVert \cdot \rVert_\E$ is mutually equivalent with a norm made from the polynomial coefficients, so Problem~\ref{prob3} is an extremal problem for continuous maps on the compact unit ball $\lVert \cdot \rVert_\E \le 1$.

Next, we describe the behavior of extremal functions under $\PSL(2,\bbR)$ transformations. This will require the following claim from \cite{SodYud93}, for which we provide a short proof.

 \begin{lemma}\label{lem:crossratio}
For every $z_0 \in \bbC \setminus \bbR$ and $x \in \overline\bbR$, there exists $t\in \overline\R$ such that
$\max_{z\in \overline\R}\left|\frac{(z-t)(x-z_0)}{(z-z_0)(x-t)}\right|=1$ and
$z=x$ is the unique maximum.
\end{lemma}

\begin{proof}
  Let $f$ be a M\"obius transformation mapping $\overline{\R}$ to
  $\partial\mathbb{D}$ with $f(z_0)=0$. Since M\"obius transformations preserve cross-ratios,
  \[
  \left|\frac{(x-z_0)(z-t)}{(x-t)(z-z_0)}\right|=\left|\frac{f(x)(f(z)-f(t))}{(f(x)-f(t))f(z)}\right|=\left|\frac{f(z)-f(t)}{f(x)-f(t)}\right|.
  \]
  By choosing $t$ so that $f(t) = -f(x)$, we have
\[
\left|\frac{(x-z_0)(z-t)}{(x-t)(z-z_0)}\right| =  \left|\frac{f(z)-f(t)}{f(x)-f(t)}\right|=\frac {|f(z)+f(x)|}2 \leq 1
\]
with equality if and only if $f(z) = f(x)$, i.e., $z=x$.
\end{proof}

In the next lemma, we consider the effect of a conformal transformation on the extremal problems, so we will emphasize dependencies on the poles, the point $x_*$ and the set $\E$ where appropriate.
We denote by $F_n(z,\E, D_n^\infty;x_*)$ a maximizer for \eqref{eq:ExtremalProblem}, and by $\cL(D_n^\infty)$ the space defined in \eqref{eq:LDivisor}. For a divisor $D$ and a a conformal map $f \in \PSL(2,\bbR)$ we define the pushforward $f_*D=D\circ f^{-1}$. Lemma~\ref{lem:conformalinvar} is an analog of \cite[Lemma 2.1]{EichLukYou} adapted to the $L^\infty$ extremal problem \eqref{eq:ExtremalProblem}.

We would like to claim that the extremizers move by a conformal map $f \in \PSL(2,\bbR)$ by
\begin{align*}
F_n(f(z),f(\E),f_*D_n^\infty;f(x_*)) = F_n(z,\E,D_n^\infty;x_*).
\end{align*}
However, this statement would be ambiguous until we prove uniqueness of extremizers, so we have to formulate the claim more carefully:

\begin{lemma}\label{lem:conformalinvar}
Let $f \in \PSL(2,\bbR)$ and let $F_n(z,f(\E),f_*D_n^\infty,f(x_*))$ be a maximizer of \eqref{eq:ExtremalProblem} for $f_*D_n^\infty $, $f(\E)$ and $f(x_*)$.
Then $F_n(f(z),f(\E),f_*D_n^\infty,f(x_*))$ is a maximizer for \eqref{eq:ExtremalProblem} for $D_n^\infty$, $\E$ and $x_*$.
\end{lemma}
\begin{proof}
M\"obius transformations preserve zeros and their multiplicity, i.e., for any rational function $F$ and any $w\in\overline{\bbC}$,
\begin{align*}
f_*^{-1}(F)_w=(F\circ f)_w.
\end{align*}
Therefore, since pushforwards of integral divisors are integral, it follows from \eqref{eq:LDivisor} that
  \begin{align}\label{eq:inclusion}F\in \mathcal{L}(f_*D_n^\infty)\implies F\circ f\in \mathcal{L}(D_n^\infty).\end{align}
  In particular, $F_n(f(z),f(\E),f_*D_n^\infty,f(x_*))\in \mathcal{L}(D_n^\infty)$.
  Since $F_n(z,f(\E),f_*D_n^\infty,f(x_*))$ solves the extremal problem on $f(\E)$, we have
  \[\|F_n(f(\cdot),f(\E),f_*D_n^\infty,f(x_*))\|_\E=\|F_n(\cdot,f(\E),f_*D_n^\infty,f(x_*))\|_{f(\E)}\leq 1.\]

  It remains then to show $F(z):=F_n(f(z),f(\E),f_*D_n^\infty,f(x_*))$ is an extremizer for $n,\E$, $D_n^\infty$ and $x_*$. This  will follow from showing that for $d_n>0$
  \begin{align}\label{eq:8sept1}
  r(f(z),f(x_*))^{d_n}- c_n r(z,x_*)^{d_n}\in \cL(D_n^\infty-x_*),\\
  \label{eq:8sept2} r(f^{-1}(z),x_*)^{d_n}-\frac{1}{c_n}r(z,f(x_*))^{d_n}\in \cL(f(D_n^\infty)-f(x_*)),
  \end{align}
  for constants $c_n>0$. Indeed, given \eqref{eq:8sept1}, \eqref{eq:8sept2}, we suppose for the sake of contradiction
  there is
  a $\tilde F\in \cL(D_n^\infty)$ with $\Re \lim_{x\to x_*} \frac{\tilde F(x)}{r(x,x_*)^{d_n}}>\Re \lim_{x\to x_*} \frac{F(x)}{r(x,x_*)^{d_n}}$.  Then, since $\tilde F\circ f^{-1}\in \cL(f_*D_n^\infty)$ by \eqref{eq:inclusion} and
  \[
  \|\tilde F\circ f^{-1}\|_{f(\E)}\leq 1,
  \]
 we contradict extremality of $F(z)$.

To show \eqref{eq:8sept1} and \eqref{eq:8sept2}, we note that for the inversions $z\mapsto -\frac1z$ and the affine transformations $z\mapsto az+b$, $b\in \R$ and $a>0$, \eqref{eq:8sept1} and \eqref{eq:8sept2} follow by elementary computations.
  Since these generate the group $\PSL(2,\bbR)$, by writing $f$ in this group as $f=f_1\circ f_2\circ f_3$, with $f_1,f_3$ affine, $f_2$ an inversion, and applying the argument immediately above three times, we have Lemma~\ref{lem:conformalinvar}.
\end{proof}

Before we state one of the main theorems of the section, we recall that the set $\supp (f)_a$ is called the set of $a$-points of the function $f$. Polynomials or entire functions with real $\pm 1$-points play an important role for uniform approximation problems and in the spectral theory of self adjoint operators; cf. \cite{ErY12,MarOst75}. They are also intimately related with the notion of a set of alternation.

We will write
$$
\E=
 \overline{\bbR} \setminus \bigcup_{i} (\ba_i,\bb_i),
$$
where $(\ba_i,\bb_i)$ are the gaps of $\E$, indexed by $i$ from a countable indexing set.

\begin{theorem}\label{thm:MarkovCorrection}
	Let $F_n$ be a maximizer for Problem \ref{prob3}.
  Let $(\ba,\bb)$ be the gap containing $x_*$.
	\begin{enumerate}[(i)]
    \item\label{it:zeroreal} $F_n$ has only real generalized zeros.
    \item\label{it:real} $F_n$ is real.
    \item\label{it:altPoint} For any distinct points $x_1,x_2\in \overline{\R}$ such that $D_n^0(x_i)\geq 1$,  there is a point $y\in \E\cap(x_1,x_2)$ with $|F_n(y)|=1$.
    \item\label{it:zerosimple}$F_n$ has only simple generalized zeros, i.e., $D_n^0\leq 1$.
	\item\label{it:1zero} $F_n$ has at most one generalized zero in each gap.
	 \item\label{it:noZeroGap} $F_n$ has no generalized zeros in the gap $(\ba,\bb)$ containing $x_*$.
	 \item \label{it:unique} There is a unique extremizer $F_n$.
    \item \label{it:realpm} If $F_n$ is not constant, $\{z\in \overline \C: F_n(z)\in[-1,1]\}\subset \overline\bbR$. In particular, all $\pm 1$-points of $F_n$ lie on $\overline{\bbR}$.
    \item\label{it:ExtrMonotonic}  If $F_n$ is not constant, let $m=\deg F_n$ and let the connected components of $F_n^{-1}((-1,1))$ be called open bands of $\E_n:=F_n^{-1}[-1,1]$. Then, there are $m$ open bands on $\E_n$, $F_n$ is strictly monotonic on each of them and their endpoints account for all $\pm 1$ points.
    \item \label{it:maxatgapn}
    \[
    F_n(\ba)=(-1)^{\sum_{\bc\in(\ba, x_*)}D_n^\infty(\bc)}\quad \text{and}\quad  F_n(\bb)=(-1)^{\sum_{\bc\in[x_*,\bb)}D_n^\infty(\bc)},
    \]
    \item \label{it:onemax}
    For any gap $(\ba_i,\bb_i)$ containing a pole $\bc_i$, either $|F_n(\bb_i)|=1$ or $|F_n(\ba_i)|=1$. If $D_n^0(\bc_i)=1$, $|F_n(\bb_i)|=|F_n(\ba_i)|=1$.
 	\end{enumerate}
\end{theorem}
\begin{remark}
Note that \eqref{it:altPoint} is stronger than saying between two zeros of $F_n$, we find an extremal point on the set; this statement provides extremal points between a zero and a pole $\bc_j$ at
which $F_n$ has a reduction in order.
\end{remark}

Many of the statements in Theorem \ref{thm:MarkovCorrection} will be proved by Markov correction arguments. We will call a rational function $M$ a Markov correction term if
$MF_n\in \mathcal{L}_{n}$ and  $M(x_*)=0$. We will define the rational function $\tilde F_n=(1-\epsilon M)F_n$, and note that
\begin{align*}
\tilde m_n(x_*)=\Re\lim\limits_{x\to x_*}\frac{\tilde F_n(x)}{r(x,x_*)^{d_n}}=\Re\lim\limits_{z\to x_*}\frac{F_n(x)}{r(x,x_*)^{d_n}}=m_n(x_*).
\end{align*}
 If there exists $\e$ so that $\|\tilde F_n\|_\E<1$, then considering the rescaled function $\tilde F_n / \lVert \tilde F_n \rVert_\E\in\cL_n$, we see that $\tilde m_n(x_*)/\lVert \tilde F_n \rVert_\E>m_n(x_*)$, contradicting the extremality of $F_n$.

\begin{proof}[Proof of Theorem~\ref{thm:MarkovCorrection}]
All the conclusions are invariant under $\PSL(2,\bbR)$ maps, so by Lemma~\ref{lem:conformalinvar}, it suffices to consider the case $x_* = \infty$.  In this case, $\E$ is a compact subset of $\bbR$.

\eqref{it:zeroreal}:
Suppose for the sake of contradiction that there is a generalized zero $z_0\in \bbC \setminus \bbR$. Define
\[
\tilde F_n(z)=\left(\frac{z-t}{z-z_0}\right)F_n(z)
\]
where $t$ is selected so that $\max_{z\in \overline\R}\left|\frac{z-t}{z-z_0} \right|=1$, using Lemma~\ref{lem:crossratio} for $x=\infty$. Since the maximum at $\infty$ is unique and $\E$ is compact, we have
$\|\tilde F_n\|_{\E}<1$, and by the discussion above, this would be a contradiction.

\eqref{it:real}:
Since all poles and zeros of $F_n$ are real, we may write $F_n=A \tilde F_n$, where $A\in \C$ with $\lvert A\rvert = 1$ and $\tilde F_n$ is real. It remains to show that $A \in \bbR$. Note that $\pm \tilde F_n$ are also admissible functions for the extremal problem. Since $F_n$ is extremal and $\tilde F_n$ is real, we have
\[
\Re\lim\limits_{x\to x_*}\frac{A\tilde F_n(x)}{r(x,x_*)^{d_n}}= \Re\lim\limits_{x\to x_*}\frac{ F_n(x)}{r(x,x_*)^{d_n}}\ge \pm \Re\lim\limits_{x\to x_*}\frac{\tilde F_n(x)}{r(x,x_*)^{d_n}} = \pm \lim\limits_{x\to x_*}\frac{\tilde F_n(x)}{r(x,x_*)^{d_n}}.
\]
Since  $\Re\lim_{x\to x_*}\frac{ F_n(x)}{r(x,x_*)^{d_n}} \neq 0$, we conclude that $\lvert \Re(A)\rvert \ge 1$ and therefore $A \in \{1,-1\}$.

\eqref{it:altPoint}: We have $\sup_{\E\cap (x_1,x_2)}|F_n|=\sup_{\E\cap [x_1,x_2]}|F_n| =\max_{\E\cap [x_1,x_2]}|F_n| $. Since $F_n$ is continuous on $\E$ we only have to explain the first equality. We only argue for $x_1$ since $x_2$ follows analogously. We distinguish two cases. If $F_n(x_1)=0$, then clearly the sup is not changed by adding $x_1$. If instead $D_n^\infty(x_1)>0$, then $x_1\notin \E$ and $(x_1,x_2)\cap\E=[x_1,x_2)\cap\E$.

Now, we assume for the sake of contradiction that $\max_{\E\cap [x_1,x_2]}|F_n|<1$. Recalling that $x_*=\infty$ so that $D_n^0(\infty)=0$,  define the Markov correction term
\begin{align*}
M(z;x_1,x_2) = \begin{cases}
\frac{1}{(z-x_1)(z-x_2)},& x_1<x_2\\
\frac{1}{(z-x_1)(x_2-z)}, & x_1>x_2
\end{cases}
\end{align*}
By distinguishing again the cases $F_n(x_i)=0$ and $D_n^\infty(x_i)>0$, we see that in either case $MF_n$ is continuous on $\E\cap [x_1,x_2]$. Thus, by our assumption we find $\e>0$ so that $\max_{[x_1,x_2]\cap \E}|\tilde F_n|<1$. Since on the rest of $\E$, the norm is lowered, we may conclude by contradiction.

\eqref{it:zerosimple}:  Clearly, $D_n^0(x_*)=0$. With our convention $x_* = \infty$ and by \eqref{it:zeroreal}, all generalized zeros are in $\R$.  Suppose $x\in \R$ with $D_n^0(x)\geq 2$.
First, we take $x\not\in \E$. We define the Markov correction term
$M(z,x)=\frac{1}{(z-x)^2}$.
If
$x\notin \E$, $z\to M(z,x)$ is continuous on $\E$ and so we may find an $\epsilon>0$ such that
$\|\tilde F_n\|_{\E}<1$.
If instead, $x\in \E$, then we conclude as in \eqref{it:altPoint} by continuity of $MF_n$ that we may find
a small enough $\epsilon>0$ so that $\|\tilde F_n\|_\E<1$.

\eqref{it:1zero}:
It follows from \eqref{it:altPoint} that between any two generalized zeros there must be a point in $\E$.

\eqref{it:noZeroGap}: Assume there is a zero in $\R\setminus [\bb,\ba]$. We use the Markov correction term
\[
M(z;x)=\begin{cases}
\frac{1}{z-x},& x<\bb\\
\frac{1}{x-z},& x>\ba
\end{cases}
\]
which is continuous and strictly positive on $\E$. By continuity and compactness, for all small enough $\epsilon > 0$, $\lVert 1 - \epsilon M \rVert_\E < 1$, so $\tilde F_n = (1 - \epsilon M) F_n$ once again contradicts extremality.

\eqref{it:unique}: Assume that there are two extremizers $F_n^1, F_n^2$. By convexity, $T_n=\frac12(F_n^1+ F_n^2)$ is then also an extremizer. Let $y_i\in \E$ be the points given by \eqref{it:altPoint} with $|T_n(y_i)|=1$.
We note that by \eqref{it:zerosimple} there are $n$ such points.
Then since $|F_n^1(y_i)|,|F_n^2(y_i)|\leq 1$ and $|T_n(y_i)|=1$,
$F_n^1(y_i)=F_n^2(y_i)=T_n(y_i)$ so that $F_n^1(y_i)- F_n^2(y_i)=0$. Define $H_n=F_n^1-F_n^2$ and let $D_n^0$ denote its divisor of generalized zeros. Then $D_n^0(x_*)\geq 1$ and $D_n^0(y_j)\geq 1$ and we conclude that $\deg D_n^0\geq n+1$. Since $H_n\in\cL_n$, this implies $H_n\equiv 0$ and $F_n^1= F_n^2$.

\eqref{it:realpm}: We write $F_n$ in reduced form as $F_n=\frac{P}{Q}$, with $\deg(P)=m$ and note that $\deg Q\leq m$ so that $\deg (F_n) = m$. If $F_n$ is nonconstant, we use a counting argument.

Take two consecutive zeros of $F_n$, $x_1$ and $x_2$. If there is no pole between them, there must be a critical value $y$ and by \eqref{it:altPoint}, it must obey $|F_n(y)|\geq 1$. Separating cases by whether $\lvert F_n(y) \rvert = 1$, we either obtain an (at least) double zero of $F_n^2 - 1$ at $y$, or zeros on intervals $(x_1,y)$ and $(y,x_2)$. Similarly, if there is a pole $y \in (x_1,x_2)$, by continuity there are $\pm 1$-points on intervals $(x_1,y)$ and $(y,x_2)$.

Thus, counted with multiplicity, there are at least two $\pm 1$-points on this interval. The $m$ simple zeros of $P$ partition $\overline{\bbR}$ into $m$ such intervals, so we have at least $2m$ total $\pm 1$-points. Since $\deg (F_n) = m$, this construction gives all the $\pm 1$-points of $F_n$. In particular, this now also holds for the set of $\pm a$-points for any $a\in[-1,1]$.

\eqref{it:ExtrMonotonic}: Let $I_n^k$ be the connected components of the open set $F_n^{-1}((-1,1))$. The previous argument shows that for $a\in (-1,1)$, the  $\pm a$-points are simple. Thus, if $F_n(x)=\pm a$, then
$F_n'(x)\neq 0$, so by continuity, the derivative has the same sign on each open band $I_n^k$. In particular, $F_n(I_n^k)=(-1,1)$ for each $k$ and there are $m$ connected components, $F_n^{-1}((-1,1)) = \cup_{k=1}^m I_n^k$. That the endpoints of $I_n^k$ account for all $\pm 1$ points follows from the counting above.

\eqref{it:maxatgapn}: First we show the modulus is $1$ at each point.  If $F_n\equiv 1$, this is clear. If $\deg(F_n)\geq 1$, we will make use of the zeros of $F_n$. Suppose for the sake of contradiction that $|F_n(\bb)|<1$. Then, define
$x:=\min\{ y:F_n(y)=0\}$, with $x\geq \bb$ by \eqref{it:noZeroGap}. We have $\sup_{z\in [\bb,x]}|F_n(z)|<1$ by \eqref{it:ExtrMonotonic}.
Define the Markov correction term $M(z,x)=\frac{1}{x-z}$ and note that $M\leq 0$ on $[\bb,x]$. By the same arguments as \eqref{it:altPoint} we derive a contradiction.
The same argument at $\ba$ shows $|F_n(\ba)|=1$.

By \eqref{it:noZeroGap}, the sign changes on $(\ba,\infty)$ can only occur at the poles contained in this interval, which we order as $\bc_{n_1}<\cdots<\bc_{n_m}$. By
\eqref{it:noZeroGap}, $F_n$ has no reduction of order at the poles at the $\bc_{n_i}$, so for a $t\in (\bc_{n_m},\infty)$, $\sgn(F_n(\ba))=(-1)^{\sum_{\bc\in(\ba, \infty)}D_n^\infty(\bc)}\sgn(F_n(t))$. By our definition of $r(z,\bc)$, $F_n>0$ on $(\bc_{n_m},\infty)$.
Since $|F_n(\ba)|=1$ by our work above, this proves the claim at $\ba$.
Similar analysis at $\bb$, with the modification that the parity of $d_n$ contributes to the sign, completes the proof.

\eqref{it:onemax}: If $F_n$ is constant, $F_n\equiv 1$ and the claim is clear. Thus, we take $F_n$ nonconstant. By \eqref{it:maxatgapn}, it suffices to consider gaps $(\ba_i, \bb_i) \neq (\ba,\bb)$.  If $|F_n(\ba_i)|=1$ there is nothing to prove. If $|F_n(\ba_i)|<1$, it follows  from monotonicity on the bands and $\lim_{x\to\infty}F_n(x)>1$ that there is a $\tilde x_1<\ba_i$ with $F_n(\tilde x_1)=0$. Similar considerations hold for $\bb_i$. If $\max\{|F_n(\ba_i)|,|F_n(\bb_i)|\}<1$, let $x_1:=\max\{ y:y< \bc_i,D_n^0(y)=1\}$ and $x_2:=\min\{ y:y>\bc_i,D_n^0(y)=1\}$. By \eqref{it:altPoint} there must be $y\in(x_1,x_2)\cap \E$, with $|F_n(y)|=1$. As in the proof of \eqref{it:maxatgapn}, we conclude from monotonicity on the bands that that either $y=\ba_i$ or $y=\bb_i$. If $D_n^0(\bc_i)=1$, we conclude in the same way that there is $y_1\in(x_1,\bc_1)$ and $y_2\in(\bc_1,x_2)$ with $|F_n(y_j)|=1$ and finally that $y_1=\ba_i$ and $y_2=\bb_i$.
\end{proof}

\begin{theorem}\label{thm:altset2}
Let $F \in \cL_n$ be real and $D_n^0$ its generalized zero divisor. Then, any set of alternation points  has at most $n+1-D_n^0(x_*)$ points.
 \end{theorem}

 \begin{proof}
Set $m=D_n^0(x_*)\geq 0$ and let $y_j\in \overline{\R}\setminus\{x_*\}$ be the $k$ points with
  $D_n^0(y_{j})>0$, where regardless of its multiplicity each point appears only once. Since $\deg D_n^0=n$, we see that $k\leq n-m$. Adding $x_*$ to this list, we cyclically order the points as $(x_*,y_1,\dots,y_{k})$. We note that these points cannot be part of an alternating set, as they either are zeros of $F$, or coincide with some $\bc_i\notin\E$ or $x_* \notin\E$. We also write $y_0 = y_{k+1} = x_*$.

Fix $1 \le j \le k+1$. On the interval $(y_{j-1},y_j)$, $F$ has no generalized zeros, so the sign changes of $F$ only occur at poles, according to the divisor $D_n^\infty$: if $(x_1, x_2) \subset (y_{j-1}, y_j)$, and $x_1, x_2$ are not poles, then
\begin{align}\label{eq:interlacing}
F(x_2) = (-1)^{\sum_{\bc\in(x_1,x_2)}D_n^\infty(\bc)} F(x_1) =(-1)^{S_n(x_2)-S_n(x_1)} F(x_1).
\end{align}
Thus, $x_1, x_2$ cannot be two consecutive points of the same alternation set, because by the definition of alternation set, this would imply $F(x_2) = (-1)^{1+S_n(x_2) - S_n(x_1)} F(x_1)$ and lead to contradiction. Thus, any alternation set has at most one point in each interval $(y_{j-1}, y_j)$ for $1\le j \le k+1$, so any alternation set has at most $k+1\leq n-m+1$ alternation points.
\end{proof}

The above theorem justifies the following definition.

\begin{definition}
We say that $F_n$ has a maximal set of alternation points if it has a set of alternation points of size $n+1$.
\end{definition}

\begin{theorem}\label{thm:altset}
If $F_n$ is the maximizer for \eqref{eq:ExtremalProblem}, then it has a maximal set of alternation points.  \end{theorem}

\begin{proof}
Due to Theorem~\ref{thm:MarkovCorrection}\eqref{it:noZeroGap}, $D_n^0(y)=0$ for all $y\in(\ba,\bb)$ and therefore, using \eqref{eq:degDivZero} and  Theorem~\ref{thm:MarkovCorrection}\eqref{it:real},\eqref{it:zerosimple}, there is a cyclically ordered sequence $(\bb, y_1,\dots,y_n,\ba)$,
so that $D_n^0(y_i)=1$. By Theorem~\ref{thm:MarkovCorrection}\eqref{it:altPoint}, for $2\leq j\leq n$, there is a point
$x_j\in(y_{j-1}, y_j)$ and $x_j\in\E$, so that $|F_n(x_j)|=1$. We claim that together with $x_{n+1}=\ba$ and $x_{1}=\bb$ these points form a maximal set of alternation points.

We start with $x_{n+1}$ and $x_1$.
Let $(\ba,\bb)$ be the gap containing $x_*$.  We have
	\begin{align*}
	S_n(\ba)=\sum_{\bc\in(\ba,x_*)}D_n^\infty(\bc).
	\end{align*}
Thus, it follows directly from Theorem~\ref{thm:MarkovCorrection}\eqref{it:maxatgapn}  that $x_{n+1}=\ba$ is an alternation point.
	Similarly, we see that
	\begin{align*}
	S_n(\bb)=\sum_{\bc\in(\bb,x_*)}D_n^\infty(\bc)
	\end{align*}
and therefore since $\deg D_n^0=\sum_{\bc}D^\infty_n(\bc)=n$ and $D_n^\infty(\bb)=0$,
\begin{align*}
n+1-1-S_n(\bb)=\sum_{\bc}D^\infty_n(\bc)-\sum_{\bc\in(\bb,x_*)}D^\infty_n(\bc)=\sum_{ \bc\in[x_*,\bb)}D^\infty_n(\bc).
\end{align*}
Thus, again by Theorem \ref{thm:MarkovCorrection}\eqref{it:maxatgapn}, also  $x_1=\bb$ is an alternation point in the above sense.

Now for $j\geq 1$ take $x_j,x_{j+1}$ and $y_j\in(x_j,x_{j+1})$ and assume that $x_j$ is an alternation point. Note that all sign changes of $F_n$ correspond either to a pole of $F_n$ or to $y_j$. Thus,
\begin{align}\label{eq:Dec14}
F_n(x_j)=(-1)^{1+\sum_{\bc\in(x_j,x_{j+1})}D_n^\infty(\bc)}F_n(x_{j+1}).
\end{align}
This is easily seen if $D_n^\infty(y_j)=0$. If $D_n^\infty(y_j)>0$, then $(F_n)_\infty(y_j)=D_n^\infty(y_j)-1$ and \eqref{eq:Dec14} still holds.
On the other hand
\begin{align*}
S_n(x_{j})-S_n(x_{j+1})=\sum_{\bc\in(x_j,x_{j+1})}D_n^\infty(\bc).
\end{align*}
Therefore, $x_{j+1}$ is also an alternating point. Thus, by induction we conclude that $\{x_i\}_{i=1}^{n+1}$ form a maximal set of alternation points for $F_n$.
\end{proof}

We also have a form of converse to Theorem~\ref{thm:altset}, which we prove as the following theorem.
\begin{theorem}\label{thm:altsetconverse}
If $F\in \cL_{n}$ is real and has a maximal alternation set, then $F$ is the unique maximizer for Problem~\ref{prob3}.
\end{theorem}

\begin{proof}
Let  $F\in \cL_n$ be real and suppose that it has a maximal set of alternation points $\{x_1,\dots,x_{n+1}\}$. By relabeling, we assume the cyclic ordering $(x_*,x_1,\dots,x_{n+1})$. By Theorem~\ref{thm:altset2}, if $F$ has an alternation set with $n+1$ points, then $(F)_\infty(x_*)=d_n$. Therefore, we  can define $\lim_{x\to x_*}F(x)/r(x,x_*)^{d_n}=:\alpha_n\in\R\setminus\{ 0\}$. It is convenient to rephrase our extremal problem: $F_n$ solves \eqref{eq:Linftyextremal} if and only if
$\tilde F_n:=\frac{1}{\lambda_n}F_n$ solves the dual problem
\begin{align}\label{eq:Linftydual}
\inf\{ \| \tilde F_n\|_\E: \lim\limits_{x\to x_*}\frac{\tilde F_n(x)}{r(x,x_*)^{d_n}}=1,\quad  \tilde F_n\in\cL_{n}\}.
\end{align}
By this duality and Theorem~\ref{thm:MarkovCorrection}\eqref{it:unique}, it will suffice to show $\tilde F:=\frac{1}{\alpha_n}F$ is also an extremizer for \eqref{eq:Linftydual}; $\|\tilde F\|_\E=\|\tilde F_n\|_\E$. Suppose that
$\| \tilde F\|_\E>\| \tilde F_n\|_\E$. We define $\tilde H_n= \tilde F- \tilde F_n$ and denote its generalized zero divisor by $D_n^0$. Our normalization implies that $D_n^0(x_*)\geq 1$. Since $\sgn(H_n(x_j))=\sgn(F(x_j))$, we have $\sgn(H_n(x_j))=(-1)^{n+1-j-S_n(x_j)}$ for $1\leq j\leq n+1$. By the computation \eqref{eq:interlacing}, we conclude that there must be $y_j\in(x_j,x_{j+1})$ with $D_n^0(y_j)\geq 1$ for $1\leq j\leq n$. Thus, $\deg (D_n^0)\geq n+1$, which contradicts $H_n\in\cL_{n}$.
\end{proof}

In particular, the proof of Theorem~\ref{thmAlternation} is now complete and we may prove Corollary~\ref{cor:gapchange}.
\begin{proof}[Proof of Corollary \ref{cor:gapchange}]
We let $\{x_1,\dots,x_{n+1}\}$ be an alternation set for $F_n$ and the point $x_*$, with cyclic ordering $(x_*,x_1,\dots,x_{n+1})$, where we recall that $x_1=\ba$ and $x_{n+1}=\bb$. By definition of $S_n$ we see that for $1\leq \ell\leq n$ we have
\begin{align}\label{eq:1Dec14}
\frac{F_n(x_\ell)}{F_n(x_{\ell+1})}=(-1)^{1+\sum_{\bc\in(x_\ell,x_{\ell+1})}D_n^\infty(\bc)}.
\end{align}
However,
\begin{align}\label{eq:2Dec14}
\frac{F_n(\ba)}{F_n(\bb)}=(-1)^{\sum_{\bc\in(\ba,\bb)}D_n^\infty(\bc)},
\end{align}
which is easier to see by using the expressions in  Theorem \ref{thm:MarkovCorrection}\eqref{it:maxatgapn}. The difference between \eqref{eq:1Dec14} and \eqref{eq:2Dec14} is manifested in the fact that $S_n$ is anchored at $x_*\in(\ba,\bb)$. Moreover, by Theorem \ref{thmAlternation}, if there exists a set $\{x_1,\dots, x_{n+1}\}$ which can be cyclically ordered  so that $F_n$ satisfies \eqref{eq:1Dec14} and \eqref{eq:2Dec14}, then for any $\tilde x_*\in(\ba,\bb)$ up to a factor $\pm 1$, $F_n$ is the maximizer of \eqref{eq:ExtremalProblem}.

Denote by $x_j^*$ a point in the gap $(\ba_j,\bb_j)$. Let $S_n^j(x):=\sum_{\bc\in \overline\R\setminus \{x_*^j\}}D_n^\infty(\bc)\chi_{[x_*^j,\bc)}(x)$. There is $1\leq k\leq n$ so that $x_*^j\in(x_k,x_{k+1})$. Let us order the $y_i$ with $D_n^0(y_i)=1$ cyclically as $(\ba,\bb,y_1,\dots,y_{n})$. The assumption $D_n^0=0$ on $(\ba_j,\bb_j)$ implies $\ba_j,\bb_j\in(y_i,y_{i+1})$ for some $1\leq i\leq n-1$. By \eqref{it:ExtrMonotonic},
$\ba_j$ and $\bb_j$ are the only points in $(y_i,y_{i+1})$ with $|F_n|=1$, and since there is exactly one of the $x_i$ in each of the $(y_i,y_{i+1})$, one and only one of $\ba_j$ and $\bb_j$ is in the alternation set $\{x_1,\dots,x_{n+1}\}$. Without a loss of generality we take $x_k=\ba_j$. We now claim the set $\{ x_1,\dots, x_k,\bb_j,x_{k+1},\dots x_{n}\}$ will form our alternation set. Since $S_n^j$ is now anchored at $x_j^*\in(\ba_j,\bb_j)$, we need to check \eqref{eq:2Dec14} for the gap $(\ba_j,\bb_j)$. From the assumption that $D_n^0=0$ on $(\ba_j,\bb_j)$ it follows that
\begin{align*}
\frac{F_n(\ba_j)}{F_n(\bb_j)}=(-1)^{\sum_{\bc\in(\ba_j,\bb_j)}D_n^\infty(\bc)}.
\end{align*}
By the assumption that $\{x_1,\dots,x_{n+1}\}$ form an alternation set (for $S_n$), \eqref{eq:1Dec14} (for $S_n^j$) is clearly satisfied for $\{x_1,\dots,x_k\}$ and for $\{x_{k+1},\dots,x_n\}$.  Using again that $\{x_1,\dots,x_{n+1}\}$ form an alternation set and that $\ba_j=x_k$, we have
\begin{align*}
F_n(x_{k+1})=(-1)^{1+\sum_{\bc\in(x_k,x_{k+1})}D_n^\infty(\bc)}F_n(x_{k})&=(-1)^{1+\sum_{\bc\in(x_k,x_{k+1})}D_n^\infty(\bc)}(-1)^{\sum_{\bc\in(\ba_j,\bb_j)}D_n^\infty(\bc)}F_n(\bb_j)\\
&=(-1)^{1+\sum_{\bc\in(\bb_j,x_{k+1})}D_n^\infty(\bc)}F_n(\bb_j).
\end{align*}
Thus, \eqref{eq:1Dec14} is also satisfied for $x_{k+1}$ and $\bb_j$. Similarly we can check \eqref{eq:1Dec14} for $x_n$ and $x_1$ and conclude that up to a factor of $\pm 1$ $F_n$ is also extremal for $x_j^*$.
\end{proof}

\begin{remark}
In the above argument, one could have removed $x_1$ and kept $x_{n+1}$ to form an alternation set for $x_*^j$.
\end{remark}

Next, we describe when the extremizer is constant:

\begin{proof}[Proof of Theorem~\ref{thmConstant}]
 Suppose $\E$ takes the above form. Without a loss of generality we assume that $(x_*,\bc_1,\dots,\bc_n)$ are cyclically ordered.
Then, $(\bb_0,\ba_0,x_1,\dots,x_{n-1})$, where $x_\ell\in \E\cap(\bc_{\ell},\bc_{\ell+1})$ for $1\leq \ell\leq n-1$  forms a set of alternation for $F_n\equiv 1$. By Theorem~\ref{thm:altsetconverse}, $F_n$ is the maximizer for \eqref{eq:ExtremalProblem}.

Suppose now the set is not of the above form. If there is a $\bc_j$ with $D_n^\infty(\bc_j)\geq 2$, by \eqref{it:zerosimple}, the extremizer $F_n$ is nonconstant. If there are two distinct poles $\bc_i$ and $\bc_j$ in a single gap, then $F_n$ cannot be constant by \eqref{it:1zero}. In either case, $F_n$ is nonconstant.
\end{proof}

We record a final corollary of Theorem~\ref{thm:altset}.

\begin{corollary}
If the extremal function $F_n$ is not constant, then $\deg F_n \ge \lceil \frac{n+1}2 \rceil$.
\end{corollary}
\begin{proof}
	By Theorem~\ref{thm:altset}, $F_n$ has at least $\lceil \frac{n+1}{2}\rceil$ points with $|F_n|=1$ with the same sign. Thus, if $F_n$ is nonconstant, it has degree at least $\lceil \frac{n+1}{2}\rceil$, and we can have at most $\lfloor\frac{n-1}{2}\rfloor$ cancellations.
\end{proof}

The set $\E_n=F_n^{-1}([-1,1])$ is called the $n$ extension of $\E$. Note that by definition it is an extension, i.e., $\E\subset \E_n$. Theorem ~\ref{thm:MarkovCorrection}, particularly the locating of $\pm 1$ points in \eqref{it:realpm} allows us to characterize this set with more specificity in the following theorem.

We recall our ternary order, and let $u_i,v_i\in \overline\R\setminus \mathring{\E}$ with $v_i\in [\ba_i,\bb_i]$ and $u_i\in [\ba_i,v_i]$.  Then
\begin{theorem}\label{thm:preimage}
For
$F_n$ nonconstant, the $n$ extension of $\E$ is of the form
\[
\E_n=\E\cup \bigcup_{i\geq 1 } [u_i,v_i]
\]
with $ [u_i,v_i]\subseteq  [\ba_i,\bb_i]$.

The following cases are possible:
\begin{enumerate}
  \item \label{it:unchanged} The gap remains unchanged, corresponding to $u_i = v_i =\ba_i$.
  \item \label{it:extended} $\E$ is extended on one edge, corresponding to
  $\ba_i=u_i$ and $v_i\ne u_i$, $v_i\ne \bb_i$, or on the other side, $v_i=\bb_i$ and $u_i\ne \ba_i$, $u_i\ne v_i$.
  \item \label{it:internal} An internal interval is added, corresponding to $[u_i,v_i]\subset (\ba_i,\bb_i)$, $u_i\ne v_i$.
  \item \label{it:closed} The gap $(\ba_i,\bb_i)$ may close, corresponding to
  $\ba_i=u_i $ and $\bb_i=v_i$.
\end{enumerate}
Moreover, in the following cases there is not extension into a gap:
\begin{enumerate}[(i)]
	\item \label{it:EnExtremalGap} If $x_* \in(\ba_i,\bb_i)$, then this gap remains unchanged, i.e., $u_i = v_i =\ba_i$.
	\item\label{it:EnDegeneration} If there is a pole $\bc_i\in  (\ba_i,\bb_i)$ and $D_n^0(\bc_i)=1$, then this gap remains unchanged, i.e., $u_i = v_i =\ba_i$.
\end{enumerate}
\end{theorem}

\begin{remark}
  \begin{enumerate}[(i)]
	\item As we will see in the proof, for gaps $(\ba_i,\bb_i)$ containing poles of $F_n$, which is guaranteed for $D_n^\infty(\bc_i)\geq 2$ by \eqref{it:zerosimple}, only the first three behaviors
	are possible.
	\item If there is an interval added to $\E_n$ as in \eqref{it:internal} above,
  then this is always related to a zero $x_i$ of $F_n$ and moreover $|F_n(\ba_i)|=|F_n(\bb_i)|=1$. Clearly, if this zero approaches a pole, the interval around it becomes smaller. In this sense \eqref{it:EnDegeneration} of the above theorem can be viewed as a limit of such situations, where the additional interval degenerates to a point.
  \end{enumerate}
\end{remark}

\begin{proof}
Applying conformal invariance of the setting, we assume again that $x_*=\infty$ and $\E$ is a compact subset of $\bbR$. Since we will prove \eqref{it:EnExtremalGap} independently, we can assume that all extensions occur in bounded gaps. We first note that any internal interval cannot degenerate to a point, i.e. when $u_i=v_i$ in \eqref{it:internal}, since due to \ref{thm:MarkovCorrection}\eqref{it:ExtrMonotonic} there are $m$ open bands and their endpoints account for all $\pm 1$ points. Thus, if the extension of the gap is not of the above form, then there would either be more then one internal interval, an extension on both sides or an extension combined with an internal interval. All cases imply that there are open bands $I_k=(y_k^-,y_k^+)$, $k=1,2$, so that $y_1^+,y_2^-\notin\E$. Let $x_k$ denote the simple zero of $F_n$ on these open bands. Using that $|F_n|<1$ on $I_k$ and $y_1^+,y_2^-\notin\E$, we see that  $\max_{[x_1,x_2]\cap\E}|F_n|<1$, contradicting Theorem \ref{thm:MarkovCorrection}\eqref{it:altPoint}.

Let us now prove \eqref{it:EnExtremalGap}: Due to  \eqref{it:maxatgapn} and \eqref{it:ExtrMonotonic} of Theorem \ref{thm:MarkovCorrection}, an extension would contain an open band that lies entirely in  $(\ba_i,\bb_i)$.
Therefore in particular, this would lead to a zero of $F_n$ on the extremal gap contradicting \eqref{it:noZeroGap} of Theorem \ref{thm:MarkovCorrection}.

It remains to prove \eqref{it:EnDegeneration}: In this case again due to \eqref{it:onemax} and \eqref{it:ExtrMonotonic}  of Theorem \ref{thm:MarkovCorrection}, this would lead to an open band  that lies entirely in  $(\ba_i,\bb_i)$ forcing $F_n$ to have an additional zero in this gap. But since already $D_n(\bc_i)= 1$, this would contradict \eqref{it:1zero} of Theorem \ref{thm:MarkovCorrection}.
\end{proof}

In the following let us assume that $F_n$ is nonconstant so that $\overline{\bbC}\setminus\E_n$ is Greenian. Note that due to Theorem \ref{thm:preimage}\eqref{it:internal}, $\E_n$ is a finite union of proper intervals and in particular is regular for the Dirichlet problem. We define
\begin{align}\label{eq:complexGreen}
B_n(z)=e^{i\phi_n}\prod_{\bc}(F_n)_\infty(\bc) B_{\E_n}(z,\bc),
\end{align}
and normalize the phase of $B_n$ by the condition
\begin{align}\label{eq:normGreen1}
\lim_{x\to x_*}B_n(x)r(x,x_*)^{d_n}>0.
\end{align}
Recall that in general $B_{\E_n}(z,\bc)$ define multivalued functions. However, we will show that their product $B_n(z)$ is in fact single valued in $\overline{\bbC} \setminus \E_n$.
\begin{theorem}\label{thm:CombRep}
$B_n$ is a single-valued analytic function on $\overline{\bbC} \setminus \E_n$ and
\begin{align}\label{eq:extremalFunctionRep}
F_n(z)=\frac{1}{2}\left(B_n(z)+\frac{1}{B_n(z)}\right).
\end{align}
\end{theorem}
\begin{proof}
Recall that $\E_n=\{z\in\overline{\bbC}: F_n(z)\in[-1,1]\}$. Therefore, since the Joukowsky map $J(\zeta)=\frac{1}{2}\left(\zeta+\frac{1}{\zeta}\right)$ maps $\bbD$ conformally onto $\overline{\bbC}\setminus [-1,1]$, the function
\begin{align*}
\Psi_n(z)=J^{-1}(F_n(z)),
\end{align*}
is well defined and single-valued in $\overline{\bbC}\setminus \E_n$. Moreover, for $x\in \E_n$, $\lim_{z\to x}|\Psi_n(z)|=1$ and $\Psi_n(z)$ has a zero of multiplicity $(F_n)_\infty(\bc)$ at each $\bc$. Thus, we conclude by the maximum principle that
\begin{align*}
-\log|\Psi_n(z)|=\sum_{\bc}(F_n)_\infty(\bc)G_{\E_n}(z,\bc)=-\log|B_n(z)|.
\end{align*}
Thus, by adding the complex conjugate, $B_n$ is defined up to a unimodular constant $c$. Finally,
\begin{align*}
0<\lim\limits_{x\to x_*}\frac{F_n(x)}{r(x,x_*)^{d_n}}=\frac{1}{2}\lim\limits_{x\to x_*}\left(\frac{cB_n(x)}{r(x,x_*)^{d_n}}+\frac{1}{cB_n(x)r(x,x_*)^{d_n}}\right)=\frac{1}{2}\lim\limits_{x\to x_*}\frac{1}{cB_n(x)r(x,x_*)^{d_n}}.
\end{align*}
Using the normalization \eqref{eq:normGreen1}, we conclude $c=1$ and obtain \eqref{eq:extremalFunctionRep}.
\end{proof}

This has the following consequence:
\begin{lemma}\label{lem:HarmonicMeasureEn}
	Let $F_n$ be represented as in \eqref{eq:extremalFunctionRep} and let $I_n$ be an open band of $\E_n$. Then
	\begin{align}\label{eq:harmonicMeasure}
	1=\sum_{\bc}  (F_n)_\infty(\bc)\omega_{\E_n}(I_n,\bc).
	\end{align}
\end{lemma}
\begin{proof}
	Recall that
	\begin{align}\label{eq:July10}
	F_n(z)=J(B_n(z))
	\end{align}
	and that $F_n$ is strictly monotonic on $I_n$. That is, either $F_n$ increases from $-1$ to $1$ or decreases from $1$ to $-1$ strictly monotonically.  Let $I_n=(a,b)$. Since $J:\partial \D\cap\bbC_\pm \to (-1,1)$ bijectively, it follows from the definition of $I_n$ and \eqref{eq:July10} that
	\begin{align*}
	|\arg B_n(b)-\arg B_n(a)|=\pi.
	\end{align*}
	By using the Cauchy-Riemann equations, we get
	\begin{align*}
	\arg B_n(b)-\arg B_n(a)=\int_a^b\frac{\partial G_{n}(x)}{\partial n}dx.
	\end{align*}
	On the other hand
	\begin{align*}
	\omega_{\E_n}(dx,\bc)=\frac{1}{\pi}\frac{\partial G_{\E_n}(x,\bc)}{\partial n}dx.
	\end{align*}
	Thus, we get
	\begin{align*}
	\arg B_n(b)-\arg B_n(a)=\pi\sum_{\bc} (F_n)_\infty(\bc)\omega_{\E_n}(I_n,\bc)
	\end{align*}
	and the claim follows.
\end{proof}

We finish this section with a Bernstein-Walsh lemma for rational functions.
\begin{lemma}\label{lem:BernsteinWalsh}
Let $K\subset \C$ be a compact, nonpolar set such that $\overline{\C}\setminus K$ is connected. Let $h$ be a meromorphic function on $\overline{\bbC}$ Then,
\begin{align}\label{eq:BernsteinWalsh1}
\frac{|h(z)|}{\|h\|_K}\leq e^{\sum_{\bc}(h)_\infty(\bc)G_K(z,\bc)}
\end{align}
If we assume in addition that $K\subset\overline{\R}$ and that $h$ is real, then
\begin{align}\label{eq:BernsteinWalsh2}
\frac{|h(z)|}{\|h\|_K}\leq \frac{1}{2}\left(e^{\sum_{\bc}(h)_\infty(\bc)G_K(z,\bc)}+e^{-\sum_{\bc}(h)_\infty(\bc)G_K(z,\bc)}\right).
\end{align}
\end{lemma}
\begin{proof}
For \eqref{eq:BernsteinWalsh1} we follow the standard proof of the Bernstein-Walsh lemma. Set $H=h/\|h\|_K$ and consider $F(z)=\log |H(z)|-\sum_{\bc}(h)_\infty(\bc)G_K(z,\bc)$. Then, $F$ is subharmonic in $\Omega=\overline{\bbC}\setminus K$ and for q.e. $\z\in \partial \Omega$ we have $\limsup_{z\to\z}F(z)\leq 0$. Moreover, if $\cV_{\bc}$ are vicinities of the points with $(h)_\infty(\bc)>0$ and $\mathcal V=\cup_{c}\mathcal V_c$, then $\log |H(z)|$ is subharmonic on $\overline{\bbC}\setminus \mathcal V$ and thus bounded above by \cite[Theorem 2.1.2]{RansfordPotTheorie}. Since the logarithmic pole on $\mathcal V_\bc$ is canceled, $F$ is also bounded above on $\cV$ and we conclude from the maximum principle \cite[Theorem 8.1]{GarMarHarmonic} that $F(z)\leq 0$ in $\Omega$.

Assume that $H$ is real and that $K$ is real. Define
$
K_{H}=\{z\in\overline{\bbC}: H(z)\in[-1,1]\},
$
but note that $K$ is not necessarily a subset of $\bbR$. However, using that $H$ is real, we have that $K\subset K_{H}$. Now, as in the proof of Theorem ~\ref{thm:CombRep} we see that
\begin{align*}
H(z)=\frac{1}{2}\left(e^{G_{H}(z)+i\widetilde{G_H(z)}}+e^{-(G_H(z)+i\widetilde{G_H(z)})}\right),\quad G_H(z)=\sum_{\bc}(h)_\infty(\bc)G_{K_H}(z,\bc).
\end{align*}
Let us also put $G(z)=\sum_{\bc}(h)_\infty(\bc)G_{K}(z,\bc_\ell)$.
Then it follows from the monotonicity of Green functions with respect to the domain that for $z\in\overline{\bbC}\setminus K_H$, we have
\begin{align*}
|H(z)|=\left|\cosh\left(G_H(z)+i\widetilde{G_H(z)}\right)\right|\leq \cosh G_H(z)\leq \cosh G(z).
\end{align*}
Note that for $z\in K_H\setminus K$, $G(z)>0$ and thus \eqref{eq:BernsteinWalsh2} also holds for such $z$. This finishes the proof.
\end{proof}

We point out that \eqref{eq:BernsteinWalsh1} is an analog of the standard Bernstein-Walsh lemma, whereas \eqref{eq:BernsteinWalsh2} is a fairly recent improvement of Schiefermayr for real polynomial problems \cite{SchiefermayerJAT}. Note that this also implies that \eqref{eq:BernsteinWalsh2} holds for $x_*\in\R\setminus K$, without the extra assumption on $h_n$ to be real. This follows from Theorem ~\ref{thm:MarkovCorrection}, where we showed that the residual extremizer is always real.

\section{Root asymptotics}\label{sec:Root}

We now turn to the study of the limiting behavior of $F_n$ as $n\to\infty$. We assume \eqref{7oct1} and assume that the measures \eqref{7oct2} have a limit
\begin{align*}
\mu = \wlim_{n\to\infty}\mu_n
\end{align*}
in the topology dual to $C(\overline{\bbR})$ and that $(x_n^*)$ is a sequence in $\overline{\bbR} \setminus \E$ without accumulation points in $\E$.   Note that  \eqref{7oct1} implies $\supp\mu\cap \E=\emptyset$.
Let further $\nu_n$ be the normalized counting measure of generalized zeros of $F_n$, i.e,
\begin{align*}
\nu_n=\frac{1}{n}\sum_{x}D_n^0(x)\delta_x.
\end{align*}
Define the family of functions
\begin{align}\label{eq:definitionhn}
h_n(z)=\frac{1}{n}\log |F_n(z)|
\end{align}
and note that $h_n$ is subharmonic in $\overline{\bbC} \setminus \supp D_n^\infty$; in particular, all functions $h_n$ are subharmonic in
\begin{align*}
\Omega_\bC=\overline{\C}\setminus K_\bC.
\end{align*}

We start with an upper estimate:
\begin{lemma}\label{lem:uniformBound}
	For any $z\in\bbC\setminus\bbR$ we have
	\begin{align}\label{eq:upperEstimate}
	\limsup h_n(z)\leq \int G_\E(z,x)d\mu(x).
	\end{align}
\end{lemma}
\begin{proof}
	Due to Lemma \ref{lem:BernsteinWalsh} and the definition of $\mu_n$ we have
	\begin{align*}
	h_{n}(z)\leq \int G_\E(z,x)d\mu_{n}(x).
	\end{align*}
	On the other hand, since $\mu_n\to\mu$ and by continuity of $G_\E(z,y)$ on $K_\bC$ we have
	\begin{align*}
	\lim\limits_{n\to\infty}\int G_{\E}(z,x)d\mu_n(x)=\int G_{\E}(z,x)d\mu(x). \qedhere
	\end{align*}
\end{proof}

We continue with some facts about potentials.
\begin{lemma}\label{lem:Potentials}
	Let $\E\subsetneq\overline{\bbR}$ be closed and not polar so that $\Omega=\overline{\bbC}\setminus\E$ is Greenian and $\mu$ be a probability measure supported on $\overline{\R}$ with $\supp\mu\cap\E=\emptyset$. Then $\int G_\E(z,x)d\mu(x)$ defines a positive superharmonic function in $\Omega$ and a harmonic function in $\Omega\setminus\supp\mu$. Moreover, as a harmonic function, it has a unique subharmonic extension to $\overline{\bbC}\setminus \supp\mu$, which vanishes q.e. on $\E$.
\end{lemma}

\begin{proof}
If $\supp \mu \subset \bbR$, it follows from \cite[Theorem II.5.1]{SaffTotikLogarithmic} and the minimum principle for superharmonic functions that $\int G_\E(z,x)d\mu(x)$ defines a positive superharmonic function in $\Omega$ and a harmonic function in $\Omega\setminus\supp\mu$ that vanishes q.e. on $\E$. In particular, locally in vicinities of $\E$ it is subharmonic and vanishes away from a polar set. Thus, by \cite[Theorem 5.2.1.]{Classpotential}, for $\z\in\E$
\begin{align*}
\int G_\E(\z,x)d\mu(x)=\limsup_{z\to \z}\int G_\E(z,x)d\mu(x)
\end{align*}
defines the unique subharmonic extension to $\Omega\setminus\supp\mu$; since all claims are conformally invariant, the general case follows.
\end{proof}

\begin{lemma}\label{lem:finiteInterlacing}
	The set $K_\bC$ intersects only finitely many open gaps.
\end{lemma}
\begin{proof}
	$K_\bC$ is a closed subset of $\overline{\bbR}$, so it is compact. It is contained in $\overline{\bbR} \setminus \E$, so its cover by the open sets $(\ba_j,\bb_j)$ has a finite subcover; in other words, $K_\bC$ only intersects finitely many gaps.
\end{proof}
We obtain immediately the following corollary:
\begin{corollary}
	For $n$ sufficiently large, $F_n$ is non-constant.
\end{corollary}

\begin{proof}
Because $D_n^0\leq 1$ for any gap and there is at most one generalized zero per gap due to Theorem \ref{thm:MarkovCorrection}\eqref{it:zerosimple},\eqref{it:1zero}, the claim follows by Lemma \ref{lem:finiteInterlacing} and $\deg D_n^\infty=n$.
\end{proof}
Since we are interested in asymptotics of $F_n$ as  $n\to\infty$, we assume from now on that $F_n$ is non-constant.
\begin{lemma}\label{lem:InfintePoles}
	Fix an open set $O\subset \overline{\R}\setminus\E$ so that $\mu(O)>0$. Then
	\begin{align}\label{eq:Dec15}
	\lim_{n\to\infty}\sum_{\bc\in O}D_n^\infty(\bc)=+\infty.
	\end{align}
\end{lemma}
\begin{proof}
By definition,
	\begin{align*}
	\mu_n(O)=\frac{1}{n}\sum_{\bc\in O}D_n^\infty(\bc).
	\end{align*}
By the Portmanteau theorem, $	\liminf_{n\to\infty}\mu_n(O)\geq \mu(O)>0$, so
	\begin{align*}
	\liminf_{n\to\infty} \frac 1n \sum_{\bc\in O}D_n^\infty(\bc) > 0
	\end{align*}
	which implies \eqref{eq:Dec15}.
\end{proof}

The following analog of Koosis's formula for the Martin or Phragm\'en Lindel\"of function \cite[Theorem on page 407]{KoosisLogarithmic1} will be very useful. It was already used in \cite[Proposition 4.3]{ChriSiZiYuDuke}.

\begin{lemma}\label{lem:KoosisAnalog}
Let $\E_1\subset \E_2\subset\overline{\bbR}$ so that $\E_1$ is not polar and let $\bc\in\overline{\bbR}\setminus\E_2$. Then,
\begin{align}\label{eq:Dec9}
G_{\E_1}(z,\bc)-G_{\E_2}(z,\bc)=\int_{\E_2\setminus\E_1}G_{\E_1}(z,x)\omega_{\E_2}(dx,c).
\end{align}
\end{lemma}
\begin{proof}
Since \eqref{eq:Dec9} is conformally invariant, by applying a conformal map we can assume that $\infty\in\E_1$, i.e., $\Omega_1=\overline{\bbC}\setminus \E_1\subset\bbC$. Define also $\Omega_2=\overline{\bbC}\setminus \E_2$. Since the logarithmic pole at $\bc$ is canceled, $G_{\E_1}(z,\bc)-G_{\E_2}(z,\bc)$ defines a superharmonic function on $\Omega_1$ which is bounded. Moreover, its Riesz measure is given by $\omega_{\E_2}(dx,\bc)|\Omega_1.$ Since $\E_1\subset \E_2$ it follows by the maximum principle that $G_{\E_1}(z,\bc)-G_{\E_2}(z,\bc)\geq 0$. Thus, in particular it has a nonnegative subharmonic minorant in $\Omega_1$ and it follows by the Riesz decomposition theorem that
\begin{align*}
G_{\E_1}(z,\bc)-G_{\E_2}(z,\bc)=\int_{\Omega_1}G_{E_1}(z,x)\omega_{\E_2}(dx,\bc)+u(z),
\end{align*}
where $u$ is the greatest harmonic minorant of $G_{\E_1}(z,\bc)-G_{\E_2}(z,\bc)$. We have already seen that $u\geq 0$. On the other hand, since $\E_1$ is the boundary for $\Omega_1$ and $\Omega_2$ it follows that for q.e. $x\in\E_1$ we have
\begin{align*}
\limsup_{z\to x}  u(z)\leq \limsup_{z\to x} (G_{\E_1}(z,\bc)-G_{\E_2}(z,\bc))=0.
\end{align*}
Thus, $u$ is a bounded harmonic function in $\Omega_1$ which vanishes $q.e.$ on $\E_1$. It follows by the maximum principle \cite[Corollary 8.3]{GarMarHarmonic} that $u=0$ and we obtain \eqref{eq:Dec9}.
\end{proof}

Compared to the standard Chebyshev problem, we encounter a technical difference for residual extremal functions. Let $(\ba_i,\bb_i)$ be a gap so that \eqref{eq:Dec15} is satisfied for $O=(\ba_i,\bb_i)$. We want to estimate $G_{\E_{n}}(z,\bc)$ for $\bc\in(\ba_i,\bb_i)$. But since $(\ba_i,\bb_i)$ is not necessarily the extremal gap, there can be an extension $(u_i,v_i)$ in this gap, which intuitively makes  $G_{\E_{n}}(z,\bc)$ smaller if $[u_i,v_i]$ is close to $\bc$. However, we have already encountered in Theorem \ref{thm:preimage}\eqref{it:EnDegeneration}, that a cancellation of a pole can be regarded as a degenerated internal interval. Thus, we are led to expect that an additional interval can have no more ``effect'' than reducing the number of Green functions in the sum by one. This is the content of the following lemma:
\begin{lemma}\label{lem:removeInterval}
Let $\E_n$ and $u_i,v_i$ be defined as in Theorem \ref{thm:preimage}. Fix a gap $(\ba_i,\bb_i)$ and define $\E_n^i=\E_n\setminus(\ba_i,\bb_i)$. Let $z\in\overline{\C}\setminus \E_n$ and $(F_n)_\infty(z)=0$. Then there is a $t\in[\ba_i,\bb_i]$ such that $G_\E(t,z)=\max_{x\in[\ba_i,\bb_i]}G_\E(x,z)$ and we have
\begin{align}\label{eq:2Oct19}
\sum_{\bc}(F_n)_\infty(\bc)G_{\E_n}(z,\bc)\geq\sum_{\bc}(F_n)_\infty(\bc)G_{\E^i_n}(z,\bc)-G_\E(z,t).
\end{align}
In particular, if $z\notin(\ba_i,\bb_i)$, then
\begin{align}\label{eq:3Oct19}
\lim_{n\to\infty}\sum_{\bc}(F_n)_\infty(\bc)G_{\E^i_n}(z,\bc)=\infty\implies \lim_{n\to\infty}\sum_{\bc}(F_n)_\infty(\bc)G_{\E_n}(z,\bc)=\infty.
\end{align}
\end{lemma}
\begin{proof}
If $z\in(\ba_i,\bb_i)$, then $t=z$ and \eqref{eq:2Oct19} is trivial. Thus, let $z\notin(\ba_i,\bb_i)$.

Since $\E_n$ is a finite union of intervals it is clearly not polar and putting $\E_n\setminus \E_n^i=[u_i,v_i]$, we obtain from Lemma \ref{lem:KoosisAnalog} that
\begin{align}\label{eq:1Oct19}
G_{\E_n^i}(z,\bc)-G_{\E_n}(z,\bc)=\int_{u_i}^{v_i}G_{\E_n^i}(z,x)\omega_{\E_{n}}(\dx,\bc).
\end{align}
By \cite[Theorem 2.1.2]{RansfordPotTheorie} a subharmonic function attains its maximum on compacts and thus $t$ is well defined. Define
\begin{align*}
\rho_n(dx)=\sum_{\bc}(F_n)_\infty(\bc)\omega_{\E_n}(dx,\bc),
\end{align*}
and note that it follows from Theorem \ref{thm:preimage} and Lemma \ref{lem:HarmonicMeasureEn} that $\rho_n([u_i,v_i])\leq 1$. Moreover, by the maximum principle
\begin{align*}
G_{\E_n^i}(z,x)\leq G_{\E}(z,x).
\end{align*}
Thus,
\begin{align*}
\sum_{\bc}(F_n)_\infty(\bc)\int_{u_i}^{v_i}G_{\E_n^i}(z,x)\omega_{\E_{n}}(\dx,\bc)=\int_{u_i}^{v_i}G_{\E_n^i}(z,x)\rho_n(dx)\leq \int_{u_i}^{v_i}G_{\E}(z,x)\rho_n(dx)\leq G_\E(z,t).
\end{align*}
Combining this with \eqref{eq:1Oct19} yields \eqref{eq:2Oct19}.
\end{proof}

By the representation \eqref{eq:extremalFunctionRep}, we have
\begin{align}\label{eq:logFn}
h_n(z)=-\frac{1}{n}\log|B_n(z)|-\frac{1}{n}\log 2+\frac{1}{n}\log|1+B_n(z)^2|.
\end{align}
The next lemma shows that the the asymptotics of $h_n$ for $n\to\infty$ are determined by the term $-\frac{1}{n}\log|B_n(z)|$. In fact, we even prove a stronger statement, which will be needed in Section \ref{sec:SzegoWidom}.
\begin{lemma}\label{lem:zeroLimit}
	Uniformly on compact subsets of $\bbC\setminus\bbR$ we have
	\begin{align}\label{eq:limitZero}
			\lim\limits_{n\to\infty}\log\left|1+B_n(z)^2\right|=0.
	\end{align}
	If we pass to a subsequence such that $\lim_{\ell\to\infty} x_{n_\ell}^*=x_\infty^*$ and $(\ba,\bb)$ denotes the gap containing $x_\infty^*$, then also for $z\in(\ba,\bb)$
		\begin{align}\label{eq:limitZero2}
		\lim\limits_{\ell\to\infty}\log\left|1+B_{n_\ell}(z)^2\right|=0.
		\end{align}
\end{lemma}
\begin{proof}
Consider $B_n$ as an analytic single-valued function on $\bbC_+$ or $\C_-$ and note that $0<|B_n(z)|< 1$. Thus, $\log\left|1+B_n(z)^2\right|=\Re\log(1+B_n(z)^2)$ defines a family of harmonic functions which is uniformly bounded from above. Thus, by the Harnack principle, the family is precompact in the space of harmonic functions together with the function which is identically $-\infty$. Therefore, it suffices to show that pointwise for fixed $z$ every subsequence has a subsequence so that \eqref{eq:limitZero} holds.
 Let us pass to a subsequence so that $\lim_{\ell \to \infty} x_{n_\ell}^*=x_\infty^*$ and let $(\ba,\bb)$ denote the gap containing $x_\infty^*$. If necessary, we pass to a further subsequence so that $x_{n_\ell}\in(\ba,\bb)$ for all $\ell>0$.  Since for  $|z|<1$
 \begin{align*}
 |\log(|1+z|)|=|\Re(\log(1+z))|\leq |\log(1+z)|=\left|\int_0^1\frac{z}{1+zt}\mathrm dt \right|\leq \frac{|z|}{1-|z|},
 \end{align*}
 it suffices to show that
 \begin{align*}
 \lim\limits_{\ell\to\infty}|B_{n_\ell}(z)|=0.
 \end{align*}
By \eqref{eq:complexGreen} and \eqref{eq:ComplexGreenGeneral} this is equivalent to
\begin{align}\label{eq:1September3}
\lim\limits_{\ell\to\infty}\sum_{\bc}(F_{n_\ell})_\infty(\bc)G_{\E_{n_\ell}}(z,\bc)=+\infty.
\end{align}

Since $\E\cap K_\bC=\emptyset$, we find a gap $(\ba_i,\bb_i)$ and $\e>0$ so that $\mu((\ba_i+\e,\bb_i-\e))>0$. Thus, by Lemma \ref{lem:InfintePoles} we have
\begin{align}\label{eq:Nov17}
\lim_{\ell\to\infty}\sum_{\bc\in (\ba_i+\e,\bb_i-\e)}D_{n_\ell}^\infty(\bc)=+\infty.
\end{align}
Note that it could be that $(\ba_i,\bb_i)=(\ba,\bb)$, which causes no problems in the following.

Set $\E^{i}=\overline{\bbR}\setminus((\ba,\bb)\cup(\ba_i,\bb_i))$ and $\E_{n_\ell}^i=\E_{n_\ell}\setminus(\ba_i,\bb_i)$. By Theorem \ref{thm:preimage}\eqref{it:EnExtremalGap}, $\E_{n_\ell}^i\subset \E^i$, so the maximum principle yields
\begin{align}\label{eq:5Oct19}
G_{\E^i}(z,\bc)\leq G_{\E_{n_\ell}^i}(z,\bc).
\end{align}
Fix $z\in\bbC_+\cup\C_-\cup(\ba,\bb)$ and note that lower semicontinuity implies
\begin{align*}
0<\delta=\min_{\bc\in[\ba_i+\e,\bb_i-\e]} G_{\E^i}(z,\bc).
\end{align*}
Then, by Theorem \ref{thm:MarkovCorrection}\eqref{it:1zero}
\begin{align*}
\sum_{\bc\in(\ba_i+\e,\bb_i-\e)}(F_{n_\ell})_\infty(\bc)G_{\E^i}(z,\bc)\geq \delta  \bigg(-1+\sum_{\bc\in (\ba_i+\e,\bb_i-\e)}D_{n_\ell}^\infty(\bc)\bigg).
\end{align*}
Thus, by \eqref{eq:Nov17} we obtain
\begin{align}\label{eq:4Oct19}
\lim\limits_{\ell\to\infty}\sum_{\bc\in(\ba_i+\e,\bb_i-\e)}(F_{n_\ell})_\infty(\bc)G_{\E^i}(z,\bc)=\infty.
\end{align}
Since by positivity of the Green function
\begin{align*}
\sum_{\bc}(F_{n_\ell})_\infty(\bc)G_{\E^i}(z,\bc)\geq\sum_{\bc\in(\ba_i+\e,\bb_i-\e)}(F_{n_\ell})_\infty(\bc)G_{\E^i}(z,\bc)
\end{align*}
we obtain together with \eqref{eq:5Oct19}  that
\begin{align*}
\lim_{\ell\to\infty}\sum_{\bc}(F_{n_\ell})_\infty(\bc)G_{\E^i_{n_\ell}}(z,\bc)=\infty.
\end{align*}
By an application of Lemma \ref{lem:removeInterval} we obtain \eqref{eq:1September3} which concludes the proof.
\end{proof}

\begin{lemma}\label{lem:positive}
	For $z\in\bbC\setminus\bbR$, we have
	\begin{align}\label{eq:lowerEstimate}
	\liminf_{n\to\infty}\frac{1}{n}\log|F_n(z)|\geq 0.
	\end{align}
\end{lemma}
\begin{proof}
Fix $z\in\bbC\setminus\bbR$ and recall \eqref{eq:logFn}.
Noting that $|B_n(z)|\leq 1$, the claim follows from Lemma \ref{lem:zeroLimit}.
\end{proof}

In contrast to the classical polynomial setting, our limits will be described by the difference of two potentials, one corresponding to the zeros of $F_n$, leading to a subharmonic part and one corresponding to the poles leading to a superharmonic part. Since in the following considerations we will work with the Riesz measures for both of them, there is no natural choice of a ``coordinate system'' and it will be convenient to apply conformal maps to logarithmic potentials. For a probabiltiy measure $\nu$ with $\supp \nu\subsetneq\overline\R$ and $z_*\in\overline{\R}\setminus\supp\nu$, let us introduce the notation

\begin{align*}
	\Phi_{\nu}(z,z_*)=\int K(x,z;z_*)d\nu(x),
\end{align*}
where
\begin{align*}
K(x,z;z_*)=\begin{cases}
\log\left|	1-\frac{z-z_*}{x-z_*}\right|,\quad &z_*\neq \infty,\\
\log|z-x|,&z_*=\infty.
\end{cases}
\end{align*}
It is straightforward to see that if $z_1,z_2\in\overline\R\setminus\supp\nu$, then there is $\beta\in\bbR$ so that
\begin{align*}
 \Phi_{\nu}(z,z_1)=\beta+\Phi_{\nu}(z,z_2).
\end{align*}

\begin{lemma}\label{lem:PotentialConformal}
Let $\nu$ be a probability measure on $\overline\bbR$, $\supp \nu \subset \E$ and $f\in\PSL(2,\R)$. If $f(\infty)= \infty$, then
\begin{align*}
\Phi_{\nu}(z,z_*)=\Phi_{f_*\nu}(f(z),f(z_*)).
\end{align*}
Otherwise,
\begin{align*}
\Phi_{\nu}(z,z_*)=\Phi_{f_*\nu}(f(z),f(z_*))-\Phi_{f_*\delta_\infty}(f(z),f(z_*)).
\end{align*}
\end{lemma}
\begin{proof}
	Let us first assume that $f(\infty)=\infty$, i.e., $f(z)=az+b$ with $a\neq 0$. Then we have
	\begin{align*}
	1-\frac{z-z_*}{x-z_*}=1-\frac{f(z)-f(z_*)}{f(x)-f(z_*)}.
	\end{align*}
  Thus, the claim follows by the transformation rule for pushforward measures.

Let now $f(\infty)\neq \infty$. Since $f$ preserves cross-ratios, we get
\begin{align*}
	1-\frac{z-z_*}{x-z_*}&=\frac{x-z}{x-z_*}=
	\frac{f(x)-f(z)}{f(x)-f(z_*)}\frac{f(z_*)-f(\infty)}{f(z)-f(\infty)}\\
	&=\left(1-\frac{f(z)-f(z_*)}{f(\infty)-f(z_*)}\right)^{-1}\left(1-\frac{f(z)-f(z_*)}{f(x)-f(z_*)}\right).
\end{align*}
 Noting that $f_*\delta_{\infty}=\delta_{f(\infty)}$, again the claim follows by applying the transformation rule for pushforward measures.
\end{proof}

\begin{lemma}\label{lem:CompactnessMeasure}
	The measures $\nu_n$ are a precompact family with respect to weak convergence on $C(\overline{\bbR})$. Any accumulation point $\nu = \lim_{\ell\to\infty} \nu_{n_\ell}$ is a probability measure and $\supp\nu \subset \E$.
\end{lemma}
\begin{proof}
	Since $\deg D_n^0=n$, precompactness follows by the Banach-Alaoglu theorem and any accumulation point is a probability measure on $\overline{\bbR}$.
	Let $(\ba,\bb)$ be a connected component of $\overline{\bbR} \setminus \E$. Let us prove that $\nu((\ba,\bb)) = 0$. By M\"obius invariance, it suffices to assume that $(\ba,\bb)$ is a bounded subset of $\bbR$. Due to Theorem~\ref{thm:MarkovCorrection}~\eqref{it:noZeroGap}, there is at most one generalized zero in $(\ba,\bb)$, thus $\nu_{n_\ell}((\ba,\bb))\leq \frac{1}{n_\ell}$ and by the Portmanteau theorem  $\nu((\ba,\bb))=0$ and
	$\supp \nu \subset \E$.
\end{proof}

In the following we will need statements also for a subsequence $(h_{n_\ell})_{\ell=1}^\infty$. Therefore, for a fixed subsequence let us define
\begin{align}\label{eq:K'}
K'=\overline{\bigcup_{\ell\geq 1}\supp D_{n_\ell}^\infty},\quad\text{and}\quad \Omega_{K'}=\overline{\bbC}\setminus K',
\end{align}
so that $h_{n_\ell}$ is subharmonic on $\Omega_{K'}$ for all $\ell$. Since $\lim_{n\to\infty}\mu_n=\mu$, we have for any subsequence (and therefore any $K'$), that $\supp\mu\subset K'\subset K_\bC$ and therefore $\Omega_\bC\subset\Omega_{K'}\subset \overline{\bbC}\setminus\supp\mu$.

If $D_n^0(z_*^1)=D_n^\infty(z_*^2)=0$, then by factoring $F_n$ we see that there is $\beta_n\in \R$ so that
\begin{align}\label{eq:potentialRephn}
h_n(z)=\beta_n+\Phi_{\nu_n}(z,z_*^1)-\Phi_{\mu_n}(z,z_*^2).
\end{align}

\begin{theorem}\label{thm:June9}
Let us pass to a subsequence so that $\lim_{\ell}\nu_{n_\ell}=\nu$, $\lim_\ell x_{n_\ell}^*=x_\infty$ and $\lim_{\ell}\beta_{n_\ell}=\beta\in\bbR\cup\{-\infty,+\infty\}$. Then, in fact $\beta\in\bbR$ and for $z_*\notin K_\bC$ we have uniformly on compact subsets of $\bbC\setminus\R$
\begin{align}
\lim\limits_{\ell\to\infty}h_{n_\ell}(z)=\beta+\Phi_\nu(z,x_\infty)-\Phi_\mu(z,z_*)=:h(z).
\end{align}
In particular, $h$ extends to a positive superharmonic function on $\overline\bbC\setminus\E$ and to a subharmonic function on $\overline{\bbC}\setminus\supp\mu$.
Moreover, for q.e. every $z\in\Omega_{K'}$
\begin{align*}
\limsup\limits_{\ell\to\infty}h_{n_\ell}(z)=\beta+\Phi_\nu(z,x_\infty^*)-\Phi_\mu(z,z_*).
\end{align*}
\end{theorem}
\begin{proof}
Let $(\ba,\bb)$ denote the gap containing $x_\infty$ and let us assume that $\ell$ is big enough so that all $x_{n_\ell}^*$ are in $(\ba,\bb)$. Due to Theorem \ref{thm:MarkovCorrection}\eqref{it:noZeroGap}, $\nu_{n_\ell}((\ba,\bb))=0$. Thus, we can write
\begin{align}\label{eq:upperEnvTheorem}
h_{n_\ell}(z)=\beta_{n_\ell}+\Phi_{\nu_{n_\ell}}(z,x_\infty^*)-\Phi_{\mu_{n_\ell}}(z,z_*).
\end{align}
Since
$
K(\cdot,z,x_\infty)
$
is continuous on $\supp\nu_{n_\ell}\subset\overline{\bbR}\setminus(\ba,\bb)$ and $
K(\cdot,z,z_*)$ is continuous on $\overline{\bbR}\setminus K'$, we get
\begin{align*}
\lim_{\ell\to\infty}\Phi_{\nu_{n_\ell}}(z,x_\infty^*)=\Phi_{\nu}(z,x_\infty^*),\quad \lim_{\ell\to\infty}\Phi_{\mu_{n_\ell}}(z,x_\infty^*)=\Phi_{\mu}(z,x_\infty^*).
\end{align*}
Since, for $z_0\in \bbC_+$,  $\Phi_{\nu}(z_0,x_\infty^*), \Phi_{\mu}(z_0,z_*)\in\bbR$ the upper and lower estimates \eqref{eq:upperEstimate} and \eqref{eq:lowerEstimate} imply that $\beta\in\bbR$. In fact, convergence is uniform on compact subsets of $\bbC\setminus\bbR$: since $\supp(\nu_{n_{\ell}}),\supp(\mu_{n_{\ell}})\subset \overline{\bbR}$ for all $\ell$ and all measures are normalized,  the estimate
\[
\log \left\lvert \frac{x-z_1}{x-z_2} \right\rvert \le \log \left( 1 + \frac{\lvert z_1 - z_2 \rvert}{\dist(z_2,\overline{\R})} \right) \le \frac{\lvert z_1 - z_2 \rvert}{\dist(z_2,\overline{\R})}, \qquad z_1,z_2 \in\C\setminus \R
\]
implies uniform equicontinuity of the potentials $\int\log\left|1-\frac{z-x^*_\infty}{x-x^*_\infty}\right|d\nu_{n_\ell}(x)$ and $\int\log\left|1-\frac{z-z_*}{x-z_*}\right|d\mu_{n_\ell}(x)$ on compact subsets of $\bbC\setminus\bbR$, and the Arzel\`a--Ascoli theorem implies uniform convergence on compacts.

By applying a conformal map $f\in\PSL(2,\R)$ and Lemma~\ref{lem:conformalinvar} we assume that $\infty\in(\ba,\bb)$ so that $\E$ and $K'$ are compact subsets of $\bbR$.

We note that $\Phi_\rho$, for $\rho=\mu,\nu$, are subharmonic in $\bbC$ and harmonic in $\bbC\setminus\supp\rho$. Thus, we only need to argue why $h$ is harmonic at $\infty$. Since $\supp \mu$ and $\E$ are bounded and $\mu, \nu$ are probability measures, we have
\begin{align}\label{eq:Dez16}
\Phi_\rho(z)=\log|z|+O(1)
\end{align}
as $z\to\infty$ and therefore, $h(z)=O(1)$ there and $h$ has a harmonic extension to $\infty$.

Finally, for $z\in\Omega_{K'}\setminus\{\infty\}$, $K(\cdot,z,z_*)$ is continuous on $K'$ and thus
\begin{align}\label{eq:1Dec29}
\lim_{\ell\to\infty}\Phi_{\mu_{n_\ell}}(z,z_*)=\Phi_{\mu}(z,z_*).
\end{align}
By the upper envelope theorem for q.e. $z\in\bbC$
\begin{align}\label{eq:2Dec29}
\limsup_{\ell\to\infty}\Phi_{\nu_{n_\ell}}(z,x_\infty^*)=\Phi_{\nu}(z,x_\infty^*).
\end{align}
Combining \eqref{eq:1Dec29} and \eqref{eq:2Dec29}, for q.e. $z\in\Omega_{K'}$ we have
\[
\limsup_{\ell\to\infty}h_{n_\ell}(z)=\lim_{\ell\to\infty}\beta_{n_\ell}+\limsup_{\ell\to\infty}\Phi_{\nu_{n_\ell}}(z,x_\infty^*)-\lim\limits_{\ell\to\infty}\Phi_{\mu_{n_\ell}}(z,z_*)=h(z). \qedhere
\]
\end{proof}

\begin{lemma}\label{lemma39}
	For $z\in\bbC\setminus\bbR$, we have
	\begin{align*}
	\liminf_{n\to\infty}\frac{1}{n}\log|F_n(z)|\geq \int G_\E(z,x)d\mu(x).
	\end{align*}
\end{lemma}
\begin{proof}
By applying a conformal map $f$, we assume $\infty\in\E$ so that $\Omega\subset \bbC$. Fix $z\in \bbC\setminus\bbR$ and let $n_\ell$ be such that
\begin{align*}
\lim\limits_{\ell\to\infty}h_{n_\ell}(z)=\liminf_{n\to\infty}h_n(z)
\end{align*}
and
\begin{align*}
\lim\limits_{\ell\to\infty}h_{n_\ell}(z)=h(z)=\beta+\Phi_\nu(z,x_\infty)-\Phi_\mu(z,z_*)
\end{align*}
in the sense of Theorem \ref{thm:June9}. Thus, $h$ defines a positive superharmonic function on $\Omega$ and
\begin{align*}
-\Delta h=\Delta\Phi_\mu(z,z_*)=2\pi\mu.
\end{align*}
 By the Riesz decomposition theorem \cite[Theorem 4.4.1]{Classpotential}, we have
\begin{align*}
h(z)=\int G_\E(z,x)d\mu(x)+u(z),
\end{align*}
where $u(z)$ is the greatest harmonic minorant of $h$. Since $h$ is positive, it follows that $u\geq 0$. Thus, we obtain
\begin{align*}
h(z)\geq \int G_\E(z,x)d\mu(x)
\end{align*}
and the claim follows.
\end{proof}

We can now prove the root asymptotics of $F_n$ and convergence of generalized zero counting measures:

\begin{proof}[Proof of Theorem~\ref{thm:RootAsymptotics}]
	Root asymptotics follow by combining Lemma \ref{lem:uniformBound} and Lemma \ref{lemma39}.

By conformal invariance, we assume that $\infty\in \supp\mu$ so that $\Omega_\mu:=\overline{\bbC}\setminus\supp\mu \subset\bbC$ and $\E$ is compact in $\bbR$. Due to  Lemma \ref{lem:CompactnessMeasure} the family $\{\nu_n\}$ is precompact and we can consider a weakly convergent subsequence
 $\nu=\lim_{j\to\infty}\nu_{n_j}$. Moreover, by Lemma  \ref{lem:Potentials}, $\int G_\E(z,x)d\mu(x)$ defines a subharmonic function in
$\Omega_\mu$. Let us compute its Riesz measure. Take $\phi\in C^\infty_c(\Omega_\mu)$ and compute
\begin{align*}
\iint G_\E(z,x)d\mu(x)\Delta\phi(z)dA(z)=\iint G_\E(z,x)\Delta\phi(z)dA(z)d\mu(x)=2\pi\iint \omega_{\E}(dz,x)d\mu(x)\phi(z)dA(z),
\end{align*}
where Fubini's theorem is justified since $\supp(\phi)\subset \overline{\C}\setminus\supp(\mu)$,  $\sup_{\supp(\phi)\times \supp(\mu)}|G_\E(z,x)|<\infty$.
That is,
\begin{align*}
\frac1{2\pi}\Delta\left(\int G_\E(z,x)d\mu(x)\right)=\int  \omega_{\E}(dz,x)d\mu(x)=:\rho.
\end{align*}
Root asymptotics and Theorem \ref{thm:June9} imply that on $\bbC\setminus\bbR$
\begin{align*}
\int G_\E(z,x)d\mu(x)=\beta+\Phi_{\nu}(z,x_\infty)-\Phi_{\mu}(z,z_*).
\end{align*}
Applying the weak identity principle for subharmonic functions \cite[Theorem 2.7.5]{RansfordPotTheorie}, this equality also holds on $\Omega_\mu$. Thus, computing the distributional Laplacian on both sides yields $\nu=\rho$ and $\wlim\nu_n=\rho$.
\end{proof}

\begin{lemma}\label{lem:shrinking}
Let us fix a gap $(\ba,\bb)$ and let $[u_n,v_n]=\E_n\cap[\ba,\bb]$. Let us pass to a subsequence, such that there are limits $u_\infty,v_\infty\in[\ba,\bb]$, i.e.,
\begin{align*}
\lim\limits_{\ell\to\infty}v_{n_\ell}=v_\infty,\quad \lim\limits_{\ell\to\infty}u_{n_\ell}=u_\infty
\end{align*}
Then
\begin{align*}
u_\infty=v_\infty.
\end{align*}
\end{lemma}
\begin{proof}
 By conformal invariance we can assume that $\infty\notin(\ba,\bb)$ and consider again $\{h_n\}$ as a family of subharmonic functions in $\Omega_\bC$. We have $\E_n\cap \supp D_n^\infty=\emptyset$, since either $F_n$ has a pole at $\bc$ or if $F_n$ has a generalized zero at $\bc$ then by \eqref{it:EnDegeneration} of Theorem \ref{thm:preimage} there is no extension in this gap. Due to Theorem  \ref{thm:RootAsymptotics},
$
\lim_{n\to\infty} h_n=\int G_\E(\cdot,x)d\mu(x).
$
Assume that $v_\infty-u_\infty=\delta>0$. For any $0<\e<\d/2$, there exists $\ell_0$ such that for all $\ell>\ell_0$, we have
 \begin{align}\label{eq:1August31}
 A:=[u_\infty+\e, v_\infty-\e]\subset [u_{n_\ell},v_{n_\ell}].
 \end{align}
 Therefore, defining $K'$ as in \eqref{eq:K'}, we have
$
A\cap K'=\emptyset.
 $
 Note that first we only have empty intersection without taking the closure, but since $\e$ above can be made smaller, we also conclude that it holds for $K'$.

By Theorem \ref{thm:June9} we have for q.e. $z\in  \Omega_{K'}$
\begin{align*}
\limsup_{\ell\to\infty} h_{n_\ell}(z)=\int G_\E(z,x)d\mu(x).
\end{align*}
Since $\supp\mu\subset K'$, it follows from Lemma \ref{lem:Potentials} that  $\int G_\E(z,x)d\mu(x)>0$ for every $z\in \Omega_{K'}$ and therefore in particular for $z\in A$. On the other hand, by definition of $\E_n$ and \eqref{eq:1August31}, we have
$h_{n_\ell}(z)\leq 0$ there. Since $A$ has positive capacity, this gives a contradiction.
\end{proof}

\section{Szeg\H o--Widom asysmptotics}\label{sec:SzegoWidom}
\subsection{Asymptotics of $\log|F_n|$}
	In the following in addition to the assumptions made in Section \ref{sec:Root}, we assume that $\E$ is a regular Parreau--Widom set. Let us recall its definition. First we assume that $\E$ is regular for the Dirichlet problem. Let $z_0\in\overline\R\setminus\E$ and denote the gap containing $z_0$ by $(\ba,\bb)$. Due to regularity and concavity of the Green function, $G_\E(z,z_0)$ has exactly one critical point in each gap $(\ba_j,\bb_j)$ except in the gap $(\ba,\bb)$. Let us denote these critical points of $G_\E(z,z_0)$ by $\xi_j$. Then we call $\E$ a regular Parreau--Widom set, if
\begin{align}\label{eq:PWcondition}
\mathcal{PW}_\E(z_0)=\sum_{j} G_\E(\xi_j,z_0)<\infty.
\end{align}
It is well known that this does not depend on the choice of $z_0$; see e.g.  \cite[Chapter V]{Hasumi}.

Denote the topological circle $\bbT_j = [\ba_j, \bb_j ] /_{\ba_j \sim \bb_j}$. Since $\ba_j, \bb_j$ are Dirichlet regular points,
\[
\lim_{x\downarrow \ba_j} G_\E(z,x) = \lim_{x\uparrow \bb_j} G_\E(z,x) = 0
\]
so with the usual convention
\begin{equation}\label{19nov1}
G_\E(z,\ba_j) = G_\E(z,\bb_j) = 0,
\end{equation}
the Green function $G_\E(z,t_j)$ depends continuously on $t_j \in \bbT_j$. We also consider the compact space
\begin{align}\label{def:setDivisors}
\cD(\E)=\prod_{j=0}^{\infty} \bbT_j,
\end{align}
equipped with the product topology. As for divisors, a functional interpretation will be convenient. Thus, for an element $D\in\cD(\E)$, $D=(t_j)_{j=0}^\infty$, we also use the functional interpretation
\begin{align*}
D(x)=\sum_{j=0}^\infty \chi_{\{t_j\}}(x).
\end{align*}

We want to associate to the divisor $D_n^0$ an element $D_n\in\cD(\E)$. In principle we want to define $D_n$ as the restriction of $D_n^0$ to $\overline{\R}\setminus \E$.  Recall that due to Theorem \ref{thm:MarkovCorrection}\eqref{it:1zero} and \eqref{it:noZeroGap}, there is at most one generalized zero in each gap and no generalized zero in the gap containing $x^*$.
Since $\deg D_n^0=n$, almost all gaps do not contain a generalized zero. To overcome this, we define
\begin{align}\label{eq:DnDivisor}
D_n=(t_n^j)_{j=0}^\infty,
\end{align}
where $t^j_n=t$, if there is $t\in [\ba_j,\bb_j]$ such that  $D_n^0(t)=1$ and otherwise we define $t_n^j$ to be the coset of $\ba_j \sim \bb_j$ in $\bbT_j$. Due to \eqref{19nov1}, these choices formally complete the definition of $D_n \in \cD(\E)$ without affecting certain sums below.

In the previous section we have described root asymptotics, i.e., asymptotics of $\frac{1}{n}\log|F_n(z)|$. The following theorem describes asymptotics of $\log|F_n(z)|$ and is the key to prove Szeg\H o-Widom asymptotics in Theorem \ref{thm:SzegoWidomAsymp}.
\begin{theorem}\label{thm:harmonic}
	Let $n_\ell$ be such that $\lim_{\ell\to\infty}x^*_{n_\ell}=x^*_{\infty}\in(\ba,\bb)$ and $\lim_{\ell\to\infty}D_{n_\ell}=D$. Then for $z\in\overline{\bbC}\setminus(\overline\R\setminus(\ba,\bb))$, we have
	\begin{align}\label{eq:limitHarmonic}
	\lim_{\ell\to\infty}\left(\log|F_{n_\ell}(z)|-\sum_{\bc}D^\infty_{n_\ell}(\bc)G_\E(z,\bc)\right)=-\log 2-\sum_{t}D(t)G_\E(z,t).
	\end{align}
	Moreover, $D(\ba)=1$.
\end{theorem}
\begin{proof}
	Define
	\begin{align*}
	H_{n_\ell}(z)=\sum_{\bc}\left((F_{n_\ell})_\infty(\bc)G_{\E_{n_\ell}}(z,\bc)-D_{n_\ell}^\infty(\bc)G_\E(z,\bc)\right).
	\end{align*}Due to \eqref{eq:logFn} and Lemma \ref{lem:zeroLimit} it remains to show that
	\begin{align}\label{eq:1September4}
	\lim\limits_{\ell\to\infty}H_{n_\ell}(z)=-\sum_{t}D(t)G_\E(z,t).
	\end{align}
	Let us assume without loss of generality that all $x_{n_\ell}^*$ lie in $(\ba,\bb)$. Recall that by Theorem \ref{thm:preimage}\eqref{it:EnExtremalGap} $D_{n_\ell}(\ba)=1$, showing that $D(\ba)=1$. Moreover, Theorem \ref{thm:preimage}\eqref{it:EnExtremalGap} implies $\E_{n_\ell}\cap(\ba,\bb)=\emptyset$ and since $G_{\E}(z,\bc)-G_{\E_n}(z,\bc)\geq 0$ and $(F_n)_\infty= D_n^\infty$ on $(\ba,\bb)$, we conclude that $(-H_{n_\ell})_{\ell}$ defines a family of positive harmonic functions in $\overline{\bbC}\setminus(\overline\R\setminus(\ba,\bb))$ and is thus by the Harnack principle precompact in the space of positive harmonic functions together with the function which is identically $+\infty$ equipped with uniform convergence on compact subsets.

	Let us now turn to the other gaps. Let $[u^j_n,v^j_n]$ denote the extension in the $j$th gap of the set $\E_n$ as in Theorem \ref{thm:preimage} and consider
		\begin{align*}
		G_{\E_{n_\ell}}(z,\bc)-G_\E(z,\bc)
		\end{align*}
		as a subharmonic function in $\Omega=\overline{\bbC}\setminus \E$, which vanishes on $\E$.  Thus, by Lemma \ref{lem:KoosisAnalog}
		\begin{align*}
		G_{\E_{n_\ell}}(z,\bc)-G_\E(z,\bc)=-\sum_{j}\int_{u^j_{n_\ell}}^{v^j_{n_\ell}}G_\E(z,x)\omega_{\E_{n_\ell}}(\dx,\bc),
		\end{align*}
	 Let us define
		\begin{align*}
		\omega_{n_\ell}(dx)=\sum_{\bc}(F_{n_\ell})_\infty(\bc)\omega_{\E_{n_\ell}}(\dx,\bc),
		\end{align*}
		and recall that this is just a finite sum. We conclude that
		\begin{align}\label{eq:2September3}
		H_{n_\ell}(z)&=\sum_{\bc}(F_{n_\ell})_\infty(\bc)\left(G_{\E_{n_\ell}}(z,\bc)-G_\E(z,\bc)\right)-\sum_{\bc}\left((D_{n_\ell}^\infty(\bc)-(F_{n_\ell})_\infty(\bc))G_\E(z,\bc)\right)\nonumber\\
		=&-\sum_{j=0}^{\infty}\int_{u^j_{n_\ell}}^{v^j_{n_\ell}}G_\E(z,x)\omega_{n_\ell}(dx)-\sum_{\bc}\left((D_{n_\ell}^\infty(\bc)-(F_{n_\ell})_\infty(\bc))G_\E(z,\bc)\right).
		\end{align}

	 Due to Lemma \ref{lem:finiteInterlacing} there are finitely many gaps containing poles. So by partitioning into finitely many subsequences, we can assume that for each $j$, for all $\ell > 0$ either $D_{n_\ell}^\infty(t^j_{n_\ell})>0$ or $D_{n_\ell}^\infty(t^j_{n_\ell})=0$, i.e., in the first case $t^j_{n_\ell}$ corresponds to a pole reduction of $F_{n_\ell}$. We will show that both cases lead to the same limit.

	Let us first consider a gap $(\ba_j,\bb_j)$ so that $D_{n_\ell}^\infty(t^j_{n_\ell})=0$ and let us assume that $t^j_{n_\ell}\to t^j_{\infty}\in(\ba_j,\bb_j)$.
	Due to Lemma \ref{lem:shrinking},
	\begin{align}\label{eq:2September2}
	\lim_{\ell}u_{n_\ell}^j=\lim_{\ell}v_{n_\ell}^j=t^j_{\infty}.
	\end{align}
	In particular for $\ell$ big enough we have $[u_{n_\ell}^j,v_{n_\ell}^j]\subset(\ba_j,\bb_j)$ and it follows then from Lemma \ref{lem:HarmonicMeasureEn} that
	\begin{align*}
	\omega_{n_\ell}([u_{n_\ell}^j,v_{n_\ell}^j])=1.
	\end{align*}
	Hence,
	\begin{align*}
	\omega_{n_\ell}|[u_{n_\ell}^j,v_{n_\ell}^j]\to \delta_{t^j_{\infty}}
	\end{align*}
	and therefore
	\begin{align*}
	\int_{u^j_{n_\ell}}^{v^j_{n_\ell}}G_\E(z,x)	\omega_{n_\ell}(dx)\to G_\E(z,t_\infty^j).
	\end{align*}

	If $t^j_{\infty}=\ba_j$, using that $G_\E(z,\cdot)$ vanishes at $\ba_j$ we conclude as above that
	\begin{align*}
		\int_{u^j_{n_\ell}}^{v^j_{n_\ell}}G_\E(z,x)	\omega_{n_\ell}(dx)\to 0= G_\E(z,t_\infty^j).
	\end{align*}

	It remains to discuss the gaps where  $D_{n_\ell}^\infty(t^j_{n_\ell})=1$. Due to Theorem \ref{thm:preimage}\eqref{it:EnDegeneration},  $u_{n_\ell}^j=v_{n_\ell}^j=\ba_j$, but in this case
	\begin{align*}
	D_{n_\ell}^\infty(t^j_{n_\ell})-(F_{n_\ell})_\infty(t^j_{n_\ell})=1.
	\end{align*}
	Thus, these are exactly the terms that contribute in the second sum in  \eqref{eq:2September3}. Since $G_\E$ is continuous we conclude that $G_\E(z,t^j_{n_\ell})\to G_\E(z,t^j_\infty)$.
	Hence, if we are allowed to interchange the limit and summation in \eqref{eq:2September3}, we have proved \eqref{eq:1September4}.
	As in \cite[Chapter V]{Hasumi}, by a Harnack-type argument, the Parreau--Widom condition implies
 $\sum_{j}\sup_{x\in(\ba_j,\bb_j)}G_\E(z,x)<\infty$ and since moreover $\omega_{n_\ell}((\ba_j,\bb_j))\leq 1$ interchanging the limits is justified and we are done.
\end{proof}

\subsection{Blaschke products, character-automorphic Hardy spaces and a related $H^\infty$ extremal problem}
We will now pass from asymptotics of the superharmonic function $\log|F_n|$ to asymptotics of the rational function $F_n$. Thus, essentially in \eqref{eq:limitHarmonic} we need to add harmonic conjugates and apply $\exp$. Thus, the left-hand side in \eqref{eq:limitHarmonic} will lead to  complex Green functions
\begin{align*}
B_\E(z,\bc)=e^{-(G_\E(z,\bc)+i\widetilde{G_\E(z,\bc)})},
\end{align*}
as defined in \eqref{eq:ComplexGreenGeneral}. We have already mentioned that in general $B_\E(z,\bc)$ is a multi-valued function in $\Omega$. Let us fix a normalization gap $(\ba,\bb)$ and $z_0\in (\ba,\bb)$ and define  $\E^j=[z_0,\ba_j]\cap\E$. Let $\tilde\g_j$ be the generator of the fundamental group $\pi_1(\Omega,z_0)$, which starts at $z_0$ and passes through the gap $(\ba_j,\bb_j)$, encircling the set $\E^j$ once. If we extend $B_\E(z,\bc)$ analytically along $\g_j$, we get
\begin{align}\label{eq:complexGreen2}
B_\E(\tilde\g_j(z),\bc)=e^{2\pi i \omega_\E(\E^j,\bc)}B_\E(z,\bc).
\end{align}

When working with multi-valued functions, it is convenient to consider them as single-valued functions on the universal cover of $\Omega=\overline{\bbC}\setminus\E$. By means of the Koebe--Poincar\'e uniformization theorem, $\Omega$ is uniformized by the disk $\D$; that is, there exists a Fuchsian group $\G$ and a meromorphic function $\bz : \D \to \Omega$ with the following properties:
\begin{align*}
1. \quad & \forall z \in \Omega \;\; \exists \, \z \in \D : \bz(\z) = z, \\
2. \quad & \bz(\z_1) = \bz(\z_2) \iff \exists \, \tilde\g \in \G : \z_1 = \tilde\g(\z_2).
\end{align*}
We fix it by the normalization $\bz(0)=z_0$, $\bz'(0)>0$. For Denjoy domains the covering map can be explicitly constructed \cite[Section 4]{RubelRyff70}. Moreover, there exists a Ford fundamental domain $\cF$, so that $\bz:\cF\to\Omega$ is bijective. We denote by $\G^*$ the group of unitary characters of $\G$; that is, group homomorphisms from $\G$ into $\T := \R/\Z$.  By the covering space formalism, $\G$ is group isomorphic to the fundamental group $\pi_1(\Omega,z_0)$. For a fixed $\z_1 \in \D$ we denote by
\begin{align}\label{def:Blaschke}
b(\z, \z_1) := \prod_{\g \in \G}\frac{\g(\z_1)}{|\g(\z_1)|}\frac{\g(\z_1)-\z}{1-\overline{\g(\z_1)}\z},
\end{align}
the standard Blaschke product. Since $\Cap_\E>0$, $\G$ is of convergent type and thus the product is indeed convergent. The functions $b(\z,\z_1)$ are character-automorphic, i.e., there exists $\chi_{z_1}\in\G^*$ such that
\begin{align*}
b(\g(\z),\z_1)=e^{2\pi i\chi_{z_1}(\g)}b(\z,\z_1), \quad \forall\g\in\G.
\end{align*}
If $z_1=\bz(\z_1)$, then these Blaschke product are related to the Green function of $\Omega$, by
\begin{align*}
-\log|b(\z,\z_1)|=G_\E(\bz(\z),z_1).
\end{align*}
Thus, we can regard the multi-valued functions $B_\E(z,z_1)$ as single-valued character-automorphic function on the universal cover.
\begin{definition}
Let $f$  be analytic in $\D$. We call $f$ ($\G^*$-) character-automorphic with character $\a\in\G^*$ if
\begin{align*}
f\circ\g =e^{2\pi i \a(\g)}f,\quad \forall\g\in\G.
\end{align*}
Similarly, if $F$ is an analytic function on $\Omega$, then we call $F$ ($\pi_1(\Omega)^*$-) character-automorphic with character $\a\in \pi_1(\Omega)^*$, if
\begin{align*}
F\circ\tilde\g =e^{2\pi i \a(\tilde\g)}F,\quad \forall\tilde\g\in\pi_1(\Omega).
\end{align*}
\end{definition}
Via the covering map $\bz$, $\G^*$- and $\pi_1(\Omega)^*$-character-automorphic functions are in one-to-one correspondence. The advantage is that  $\G^*$- character-automorphic functions on the universal cover $\D$ are single-valued. Therefore, we will formulate all convergence results for the corresponding single-valued lifts on $\bbD$.

Recall that $H^\infty_\Omega(\a)$ denotes the space of bounded analytic character-automorphic functions, $F$, in $\Omega$; see \eqref{eq:HinftyNorm}.
It is a fundamental result of Widom \cite{Widom71} that if $\E$ is a Parreau--Widom set, then $H^\infty_\Omega(\a)\neq\{0\}$ for every $\a\in\pi_1(\Omega)^*$.
The Widom maximizer for $x_*$ and character $\a$ is the unique function $W(z;\a,x_*)$ in the unit ball of $H^\infty_\Omega(\a)$ such that
\begin{align}\label{eq:WidomMax}
W(x_*;\alpha,x_*)=\max\{\Re F(x_*): F\in H^\infty_\Omega(\a), \|F\|_\Omega\leq 1\}.
\end{align}

We are now ready to state a definition of Direct Cauchy theorem. It is usually stated as a point evaluation property for certain $H^1$ functions in $\Omega$ \cite{Hasumi}, and hence the name, but it can be equivalently defined by the following:
\begin{definition}
We say that the Direct Cauchy Theorem (DCT) holds in $\Omega$, if for one and hence for all $x_*\in\Omega$, the map $\a\mapsto W(x_*;\a,x_*)$ is continuous on $\pi_1(\Omega)^*$ equipped with the topology dual to the discrete topology on $\G$.
\end{definition}

Let us for notational convenience also define $B_\E(z,z_0)\equiv 1$, if $z_0\in\E$. Note that generally the harmonic conjugate is fixed up to an additive constant. So an additional normalization is required in \eqref{eq:ComplexGreenGeneral}. Since we will have varying normalizations, we will not fix it for a single function, but assume instead that for  products of complex Green functions all of them are normalized to be positive at the same point.
In this way, we can associate to any divisor $D\in\cD(\E)$ a product of complex Green functions, in other words a Blaschke product, by
\begin{align}\label{eq:BlaschkeProdDiv}
B_\E(z,D)=B_\E(z,D,\phi)=e^{i\phi}\prod_{t}D(t)B_\E(z,t).
\end{align}
Note that
\begin{align*}
-\log|B_\E(z,D)|=\sum_{t}D(t)G_\E(z,t),
\end{align*}
that is, these are exactly expression of the type appearing in \eqref{eq:limitHarmonic}. Moreover, the Widom condition guarantees that $B_\E(z,D)$ converges to a non-trivial function for any $D\in\cD(\E)$. Let us define the restriction
\begin{align*}
\cD_k(\E)=\{D\in\cD(\E): D(\ba_k)=1\}.
\end{align*}
For $D\in\cD_k(\E)$ it is natural to normalize $B_\E(z,D,\phi)$ such that $B_\E(z,D,\phi)>0$ on $(\ba_k,\bb_k)$ which we fixes $\phi$. To be more precise, since complex Green functions are defined locally and then extended analytically, this normalization holds only for one branch. Let us always assume that this branch corresponds to the values of the lift to $\D$ in the fundamental domain $\cF$.

The Abel map is an important object in the spectral theory of self adjoint difference and differential operators. It is a map $ \pi$ from Divisors $\cD_k(\E)$ to the characters $\pi_1(\Omega)^*$. However, there is a subtle difference between this Abel map and the Abel map which we will implicitly use for Problem \ref{prob3}. It can be seen from the definition of $\cD(\E)$. In spectral theory one would usually take a two-fold cover of the interval $[\ba_j,\bb_j]$ and identify the endpoints of the two copies of the interval, whereas in our case we only took one copy and identified $\ba_j\sim \bb_j$. This map $ \pi$ is also the reason why the DCT property is needed, because this assumption makes $\pi$ a bijection which is used in the proof of the following theorem. The proof relies on the fundamental construction of the generalized Abel map from Sodin and Yuditskii \cite{SoYud97}.
\begin{theorem}[{\cite[Theorem 5.1]{ChriSiZiYuDuke},\cite[Proposition 2.3]{EichYuSbornik}}]\label{eq:thmWidomMax}
Let $\Omega$ be a regular Parreau--Widom domain such that DCT holds. Let $D\in \cD_k(\E)$ and let $\a$ be the character of $B_\E(z,D)$ defined by  \eqref{eq:BlaschkeProdDiv}. Then, for $x_*\in(\ba_k,\bb_k)$ we have
\begin{align*}
W(z;\a,x_*)=B_\E(z,D).
\end{align*}
\end{theorem}

We see again that as for $F_n$, the extremal function only depends on the chosen gap $(\ba_k,\bb_k)$ and not the particular extremal point in the gap. Since the above theorem holds for any gap and arbitrary Blaschke products associated to divisors in $\cD_k(\E)$, we conclude that if $D\in\cD_j(\E)$ for $j\neq k$, then up to a unimodular constant $B_\E(z,D)$ is also the Widom maximizer for the gap $(\ba_j,\bb_j)$. This is in line with the Corollary \ref{cor:gapchange} for $F_n$.

Let $D_n^\infty$, $x_n^*$ and $d_n$ be as in Problem \ref{prob3} and define
\begin{align*}
B^{(n)}_\E(z)=e^{i\phi_n}\prod_{\bc}D_n^\infty(\bc)B_\E(z,\bc),
\end{align*}
where $e^{i\phi_n}$ is chosen such that
\begin{align}\label{eq:NormGreen2}
\lim\limits_{x\to x_n^*}B^{(n)}_\E(x)r(x,x_n^*)^{d_n}>0.
\end{align}
Let $\chi_n$ denote the character of $B_\E^{(n)}$. Let further
\[
W_n(z)=W(z;\chi_n,x_n^*),
\]
denote the Widom maximizer for the point $x_n^*$ and character $\chi_n$.

For the following we follow the spirit of \cite{ChriSiZiYuDuke} and  state convergence results on the universal cover $\D$ without introducing the corresponding lift of multi-valued functions on $\Omega$. To give an example: if $Q_n$ are $\pi_1(\Omega)^*$-character-automorphic function on $\Omega$, we will write $Q_n\to Q$ uniformly on compact subsets of $\D$, meaning that there are lifts $q_n$ of the $Q_n$ which are $\G^*$-character-automorphic functions such that $q_n\to q$ uniformly on compact subsets of $\D$ and $Q$ is the projection of $q$.

\begin{theorem}\label{thm:SzegoWidomAsymp}
Let $\E$ be a regular Parreau--Widom set, such that DCT holds in $\Omega$ and $F_n$ be the extremizer of \eqref{eq:Linftyextremal}. Then uniformly on compact subsets of $\D$, we have
\begin{align*}
\lim\limits_{n\to\infty}\left(B^{(n)}_\E(z)F_n(z)-\frac{1}{2}W_n(z)\right)=0.
\end{align*}
\end{theorem}

We will use the following simple criterion based on normality; note that it is simpler than the corresponding criterion used in the polynomial case \cite[Proposition 4.2]{ChriSiZiYuDuke}, since our approach avoids working on multivalued functions on varying domains:
\begin{proposition} \label{prop:convergence}
	Let $\{q_n\}_{n=1}^\infty$ be a normal family on $\D$.
	Let $q_\infty$ be analytic on $\D$ so that for some $\zeta_0\in \D$ and some neighborhood, $V$, of $\zeta_0$ we have that
	\begin{align}
	&\lim_{n\to\infty}|q_n(\z)|=|q_\infty(\z)|\quad \text{ for all }\z\in V; \label{31dec1} \\
	&q_n(\z_0)>0,\quad q_\infty(\z_0)>0. \label{31dec2}
	\end{align}
	Then $q_n\to q_\infty$ uniformly on compact subsets of $\D$.
\end{proposition}
\begin{proof}
By normality, it suffices to prove that any subsequence $(q_{n_\ell})_{\ell=1}^\infty$ which converges uniformly on compacts has the limit $q_\infty$. Denote by $f$ the limit of such a sequence. By \eqref{31dec1}, $\lvert f(\z) \rvert =\lvert q_\infty(\z) \rvert$ for all $\z \in V$. By \eqref{31dec2}, by possibly decreasing $V$, we can assume $q_\infty(\zeta) \neq 0$ for $\zeta \in V$, so by the maximum principle applied to $f / q_\infty$, we conclude $f = e^{i\phi} q_\infty$ for some unimodular constant $e^{i\phi}$. By \eqref{31dec2}, $f(\zeta_0) \ge 0$ and $q_\infty(\zeta_0) > 0$, so $e^{i\phi} = 1$ and $f = q_\infty$.
\end{proof}

Defining
\begin{align*}
Q_n(z)=F_n(z)B^{(n)}_\E(z),
\end{align*}
the strategy is now clear: First we need to check that $Q_n(z)$ defines a normal family. Realizing that $\log|Q_n(z)|$ is exactly the left hand-side in \eqref{eq:limitHarmonic}, Theorem \ref{thm:harmonic} and Proposition \ref{prop:convergence} imply that all accumulation points are Blaschke products. Combining this with Theorem \ref{eq:thmWidomMax} finishes the proof of Theorem \ref{thm:SzegoWidomAsymp}.
\begin{lemma}
The sequence $\{Q_n\}_{n=1}^\infty$ forms a normal family in $\D$.
\end{lemma}
\begin{proof}
	Since on $\Omega$	\begin{align*}
	(F_n)_\infty\leq D_n^\infty=(B^{(n)}_\E)_0,
	\end{align*}
	 $Q_n$ are analytic $\pi_1(\Omega)^*$-character automorphic functions in $\Omega$. They have therefore $\Gamma^*$-character automorphic lifts to $\D$. By Montel's theorem \cite[Chapter 6]{SimonBasicCompAna}, it suffices to show that $|F_nB^{(n)}_\E|\leq 1$ in $\Omega$. The functions $\log|Q_n|$ are
	 subharmonic in $\Omega$. Moreover, since $\E$ is regular and $|F_n|\leq 1$ on $\E$, for every $\zeta\in\E$
	\begin{align*}
	\limsup_{z\to \zeta}	\log|Q_n(z)|=\limsup_{z\to \zeta} \log|F_n(z)|-\lim\limits_{z\to\zeta}\sum_{t}D_n^\infty(t)\lim_{z\to \zeta}G_\E(\zeta,t)\leq 0,
	\end{align*}
	where we used Dirichlet regularity and the fact that the sum is only finite. The claim follows by the maximum principle for subharmonic functions \cite[Theorem 2.3.1]{RansfordPotTheorie}.
\end{proof}

\begin{lemma}\label{lem:subsequences}
	Let $n_\ell$ be a subsequence such that $\lim_{\ell\to\infty}D_{n_\ell}=D$ and $\lim_{\ell\to\infty}x_{n_\ell}^*=x_\infty^*\in(\ba_j,\bb_j)$ for some $j\geq 0$. Then, uniformly on compact subsets of $\D$ we have
	\begin{align*}
	\lim\limits_{\ell\to\infty}Q_{n_\ell}(z)=\frac{1}{2}B_\E(z,D),
	\end{align*}
	where $D\in\cD_j(\E)$ and $B_\E(x_\infty^*,D)>0$.
\end{lemma}
\begin{proof}
	Let us assume without loss of generality that all $x_{n_\ell}^*$ lie in $(\ba_j,\bb_j)$. By \eqref{eq:NormGreen2}, we have
	\begin{align*}
	Q_{n_{\ell}}(x_{n_\ell}^*)>0.
	\end{align*}
	Moreover, $Q_{n_\ell}$ are real, i.e., $Q_{n_\ell}(\overline{z})=\overline{Q_{n_\ell}(z)}$. Since $D_n^0(t)=0$ for every $t\in(\ba_j,\bb_j)$, it follows that $Q_{n_\ell}(t)>0$. Thus, in particular at $x_\infty^*$. Thus we can apply  Proposition \ref{prop:convergence} in a vicinity of $x_{\infty}^*$ and then the claim follows from Theorem \ref{thm:harmonic}.
\end{proof}
\begin{proof}[Proof of Theorem \ref{thm:SzegoWidomAsymp}]Since ${Q_n},W_n$ form normal families, by precompactness it suffices to prove that every subsequence has a subsubsequence so that $\lim_\ell Q_{n_\ell}-W_{n_\ell}= 0$.	 Let us pass to a subsequence such that $\lim_{\ell\to\infty}D_{n_\ell}=D$ and $\lim_{\ell\to\infty}x_{n_\ell}^*=x_\infty^*\in(\ba_j,\bb_j)$ as in Lemma~\ref{lem:subsequences}. Then by Lemma \ref{lem:subsequences}
	\begin{align*}
	\lim\limits_{\ell\to\infty}Q_{n_\ell}(z)=\frac{1}{2}B_\E(z,D).
	\end{align*}
	If $\alpha$ is the character of $B_\E(z,D)$ this implies that $\chi_{n_\ell}\to\alpha$. By Theorem \ref{eq:thmWidomMax}, $B_\E(z,D)=W(z,\alpha,x^*_\infty)$. On the other hand, it is proven in \cite[Theorem 3.1]{ChriSiZiYuDuke} that DCT implies that $W(z;\chi_{n_\ell},x^*_{n_\ell})\to W(z,\alpha,x^*_\infty)$ uniformly on compact subsets of $\D$. In this reference there is no sequence of extremal points, but since the Widom maximizer only depends on the given gap and not on the particular point, the sequence $W(z;\chi_{n_\ell},x^*_{n_\ell})$ eventually only depends on the character. This concludes the proof.
\end{proof}

\providecommand{\MR}[1]{}
\providecommand{\bysame}{\leavevmode\hbox to3em{\hrulefill}\thinspace}
\providecommand{\MR}{\relax\ifhmode\unskip\space\fi MR }
\providecommand{\MRhref}[2]{%
	\href{http://www.ams.org/mathscinet-getitem?mr=#1}{#2}
}
\providecommand{\href}[2]{#2}

	\end{document}